\documentclass[11pt]{article}

\usepackage{amsmath,amssymb,amsthm,amsfonts,enumerate,graphicx}
\usepackage{amssymb, amsfonts,color}
\usepackage{amsmath}
\usepackage{latexsym}
\usepackage{mathrsfs}
\usepackage{graphicx}
\usepackage{enumerate}
\newtheorem{theorem}{Theorem}[section]
\newtheorem{lemma}[theorem]{Lemma}
\newtheorem{proposition}[theorem]{Proposition}
\newtheorem{corollary}[theorem]{Corollary}
\newtheorem{definition}{Definition}[section]

\numberwithin{equation}{section}

\renewcommand{\baselinestretch}{1.0}

\begin{document}

\title{Open set condition and pseudo Hausdorff measure of self-affine IFSs
\footnotetext{Math Subject Classifications. 28A80, 28A78.}
\footnotetext{The research of Fu was supported by NSFC  grant 11401205.}
\footnotetext{The research of Qiu was supported by NSFC grant 11471157.}
\footnotetext{The research of Gabardo was supported by an NSERC grant.}
\footnotetext{Email:xiaoyefu@mail.ccnu.edu.cn, gabardo@mcmaster.ca, huaqiu@nju.edu.cn.}}
\author{Xiaoye Fu$^1$, Jean-Pierre Gabardo$^2$, Hua Qiu$^3$ \\
    \small {1.} School of Mathematics and Statistics\\
   \small and Hubei Key Laboratory of Mathematical Sciences, \\
    \small Central China Normal University, Wuhan 430079, P.R.China\\
     \small {2. } Department of Mathematics and Statistics, McMaster University,\\
   \small    Hamilton, Ontario, L8S 4K1, Canada\\
\small {3.} Department of Mathematics, Nanjing University, Nanjing, 210093, P.R.China}

\date{}
\maketitle

\begin{abstract}
Let $A$ be an $n\times n$ real expanding matrix and $\mathcal{D}$ be a finite subset of $\mathbb{R}^n$
with $0\in\mathcal{D}$. The family of maps $\{f_d(x)=A^{-1}(x+d)\}_{d\in\mathcal{D}}$ is called a self-affine iterated function system (self-affine IFS). The self-affine set $K=K(A,\mathcal{D})$ is the unique compact set determined by $(A, {\mathcal D})$ satisfying the set-valued equation $K=\displaystyle\bigcup_{d\in\mathcal{D}}f_d(K)$. The number $s=n\,\ln(\# \mathcal{D})/\ln(q)$ with $q=|\det(A)|$,
is the so-called pseudo similarity dimension of $K$. As shown  by He and Lau,
one can associate with $A$ and any number $s\ge 0$ a natural pseudo Hausdorff measure denoted by
$\mathcal{H}_w^s.$
In this paper, we show that, if $s$ is chosen to  be the pseudo similarity dimension of $K$, then  the condition
$\mathcal{H}_w^s(K)> 0$ holds if and only if  the IFS $\{f_d\}_{d\in\mathcal{D}}$ satisfies
the open set condition (OSC). This extends the well-known result for the self-similar case
that the OSC is equivalent to  $K$ having positive Hausdorff measure $\mathcal{H}^s$ for a suitable $s$.
Furthermore, we relate the exact value of pseudo Hausdorff measure
$\mathcal{H}_w^s(K)$ to a notion of upper
$s$-density with respect to the pseudo norm $w(x)$ associated with $A$ for the measure
$\mu=\lim\limits_{M\to\infty}\sum\limits_{d_0,\dotsc,d_{M-1}\in\mathcal{D}}\delta_{d_0 + Ad_1 + \dotsb + A^{M-1}d_{M-1}}$ in the case  that
$\#\mathcal{D}\le\lvert\det A\rvert$.
\end{abstract}

\renewcommand{\baselinestretch}{1.0}

\section{Introduction}

\begin{definition}\label{th1.1}
Let $M_n(\mathbb{R})$ denote the set of $n\times n$ matrices with real entries.
A matrix $A\in M_n(\mathbb{R})$ is called  expanding if all its eigenvalues $\lambda_i$ satisfy
$\lvert\lambda_i\rvert>1$.
A self-affine set in ${\mathbb R}^n$ is a compact set
$K\subseteq {\mathbb R}^n$ satisfying the set-valued equation $AK=\bigcup\limits_{d\in\mathcal{D}}(K+d)$, where $A\in M_n(\mathbb{R})$ is an expanding matrix and $\mathcal{D}\subseteq\mathbb{R}^n$ is a finite set of
distinct real vectors, which is called a digit set. $K$ is called a self-similar set  if $A$ is a similar matrix, i.e. $A = \rho R$, where $\rho>1$ and $R$ is an orthogonal matrix. To simplify the notations,
we let $q=|\det(A)|$.
\end{definition}

For an expanding matrix $A\in M_n(\mathbb{R})$ and a digit set $\mathcal{D}\subseteq\mathbb{R}^n$, it has been shown that the pair $(A, {\mathcal D})$ can uniquely determine a self-affine set $K:=K(A,\mathcal{D})$ (see \cite{H}). Given the pair $(A, {\mathcal D})$,
define
$$f_d(x)=A^{-1}(x+d), \ d\in\mathcal{D}.$$
The family of maps $\{f_d\}_{d\in\mathcal{D}}$ is called a \emph{self-affine iterated function system} (self-affine IFS).
An important property of these maps is that they are contractive with respect to a suitable norm on $\mathbb{R}^n$
(see \cite{LWG}). It is clear that the self-affine set $K:=K(A,\mathcal{D})$ determined by the pair $(A, {\mathcal D})$  satisfies $K=\bigcup\limits_{d\in\mathcal{D}}f_d(K)$.

\vspace{0.3cm}

\begin{definition}\label{th1.2}
For the pair $(A, {\mathcal D})$
as above,
we say that the IFS $\{f_d\}_{d\in\mathcal{D}}$ satisfies the open set condition (OSC) if there exists a non-empty bounded open set $V$ such that
\begin{eqnarray*}
\bigcup\limits_{d\in\mathcal{D}} f_d(V)\subset V \ {\rm{and}} \ f_d(V)\cap f_{d^{\prime}}(V)=\emptyset \ {\rm{for}} \ d\ne d^{\prime}\in\mathcal{D}.
\end{eqnarray*}
\end{definition}

The OSC is the most important separation condition in the theory of IFS and it is thus very useful
to find conditions equivalent to it.
When the IFS is self-similar, it is well-known \cite{S} that the OSC is equivalent to the self-similar set generated by the IFS having positive Hausdorff measure. For the self-affine case, He and Lau \cite{HL} showed that if the  OSC is satisfied, then the corresponding self-affine set has positive pseudo Hausdorff  measure.
This last measure is defined by using a pseudo norm constructed from the matrix $A$ instead of the
classical Euclidean norm. In this paper, we prove that the OSC is indeed equivalent
to the self-affine set generated by the IFS having
positive pseudo Hausdorff measure by showing that the converse also holds.

\vspace{0.3cm}
For an integer $M\ge 1$, consider the sets
\begin{eqnarray*}
{\mathcal D}_M = \bigg\{\sum\limits_{j=0}^{M-1} A^j d_j: d_j\in {\mathcal D}\bigg\}, \ \ \text{and} \ \ {\mathcal D}_{\infty} = \bigcup \limits_{M\ge 1} {\mathcal D}_M.
\end{eqnarray*}
Then ${\mathcal D}_M\subset {\mathcal D}_{M+1}$ for any $M\ge 1$ if $0\in{\mathcal D}$.
 Combining our results with those proved by He and Lau (Theorem 4.4 in \cite{HL}), we provide some
 conditions  equivalent to the OSC  for self-affine IFSs.

\begin{theorem}\label{th1.3}
The following conditions are equivalent.
\begin{enumerate}[(i)]
\item The IFS $\{f_d\}_{d\in{\mathcal D}}$ satisfies the OSC;
\item $0 < \mathcal{H}_w^s(K) < \infty$, where $s=n\,\ln(\# \mathcal{D})/\ln(q)$ and $\mathcal{H}_w^s(K)$ denotes the $s$-dimensional pseudo Hausdorff measure of $K$ generated by the IFS $\{f_d\}_{d\in{\mathcal D}}$;
\item  $\# {\mathcal D}_M = (\#{\mathcal D})^M$ and ${\mathcal D}_{\infty}$ is a uniformly discrete set, i.e. there exists $\delta > 0$ such that $\|x-y\|>\delta$ for any distinct elements $x,y$ of ${\mathcal D}_{\infty}$.
\end{enumerate}
\end{theorem}

\vspace{0.3cm}

For the proof of Theorem \ref{th1.3}, we utilize the connection between pseudo norm and Euclidean norm
as well as  the technique used by Schief \cite{S}, Bishop and Peres \cite{BP} for the self-similar case.
We also would like to mention that there have been several equivalent characterizations for the OSC under special cases given by Lagarias and Wang  (Theorem 1.1 in \cite{LWG}), by He and Lau (Theorem 4.4 in \cite{HL}) and by Fu and Gabardo (Theorem 3.2 in \cite{FG}).

\vspace{0.3cm}

In fractal geometry, one of the classical questions  is to study
the Hausdorff dimension and the corresponding Hausdorff
measure of the self-affine set
$K(A,\mathcal{D})$ determined by the pair  $(A, {\mathcal D})$.

\vspace{0.3cm}

In the case that $K(A, {\mathcal D})$ has positive Lebesgue measure and
$\#\mathcal{D} = \lvert\det A\rvert \in\mathbb{Z}$, $K$ is called
\emph{a self-affine tile} and the corresponding set ${\mathcal D}$ is
called \emph{a tile digit set}, where $\#\mathcal{D}$ denotes the
number of elements in $\mathcal{D}$. The Lebesgue measure and many
aspects of the theory of self-affine tiles including the structure
and tiling properties, the connection to wavelet theory,
the fractal structure of the boundaries and the classification of tile digit sets
have been investigated thoroughly
(see e.g. \cite{LWG, LWN, GY, GHS, LWA, GM, LW, LLR2012, LLR2013}).

\vspace{0.3cm}

The situation becomes more complicate when $\#\mathcal{D} >q:= \lvert\det A\rvert$
because the sets $K+d$, $d\in\mathcal{D}$, might overlap.
He, Lau and Rao \cite{HLR} considered the problem as to whether or not
the Lebesgue measure of $K(A,\mathcal{D})$ is positive for this case. Qiu \cite{Q}
provided an algorithm for calculating the Hausdorff measure
of a special class of Cantor sets $K(A,\mathcal{D}) \subset \mathbb{R}$ with overlaps.

\vspace{0.3cm}

It is easy to see that the Lebesgue measure of $K(A,\mathcal{D})$ is $0$ if
$\#\mathcal{D} < q$, a situation which has motivated many researchers to study the Hausdorff dimension and Hausdorff measure of such sets $K(A,\mathcal{D})$.
For self-similar sets satisfying  certain separating conditions
(e.g. open set condition \cite{F}, weak separation condition \cite{LN,LNR}, finite type condition \cite{NW}), there exist methods to calculate
their Hausdorff  dimensions \cite{F, HLR, NW, SS} and the corresponding Hausdorff   measures \cite{AS, FG, J, J2007, JZZ2002, Xiong2005, ZL2000, ZF2000}.
However, no many results are available in that direction for  self-affine sets.
The  difficulty stems from the non-uniform contraction in different directions, in contrast
to the self-similar case where
the contraction is uniform in every direction. In \cite{HL}, He and Lau
defined a pseudo norm $w(x)$ associated with  the matrix $A$ and
replaced the Euclidean norm by this pseudo norm to define the Hausdorff dimension and the Hausdorff measure for subsets in $\mathbb{R}^n$. They called these the \emph{pseudo Hausdorff dimension} $\dim_H^w$ and
the \emph{pseudo Hausdorff measure} $\mathcal{H}_w^s$, respectively. This setup gives a convenient estimation to the classical Hausdorff dimension  of $K(A,\mathcal{D})$ and, furthermore, it makes $K(A,\mathcal{D})$ have a structure similar
to that of a self-similar set since the pseudo norm defined in terms of $A$ absorbs the non-uniform contractivity from $A$.

\vspace{0.3cm}

In this paper, we are interested in the computation of the pseudo Hausdorff measure of self-affine sets in the case that $\#\mathcal{D}\le q$. This is motivated by the results in \cite{FG}, which gave an exact expression for the Lebesgue measure of $K(A,\mathcal{D})$ with $\#\mathcal{D} = q$ and the Hausdorff measure of the self-similar set $K(A,\mathcal{D})$ associated with its similarity dimension in the case that $\#\mathcal{D} \le q$.
One of the main results of this paper is to relate the pseudo Hausdorff measure of $K(A,\mathcal{D})$,
namely
$\mathcal{H}_w^s(K(A,\mathcal{D}))$ where $s=n\,\ln(\# \mathcal{D})/\ln(q)$ is the
pseudo similarity dimension of $K$,
to a  notion of upper density with respect to (w.r.t.) $w(x)$ for the measure $\mu$ which is defined by
\begin{eqnarray}\label{e-1-1}
\mu=\lim\limits_{M\to\infty}\sum\limits_{d_0,\dotsc,d_{M-1}\in\mathcal{D}}\delta_{d_0+Ad_1+\dotsb+A^{M-1}d_{M-1}}.
\end{eqnarray}

\begin{theorem}\label{th1.4}
Let $K: = K(A, \mathcal {D})$ be a self-affine set and let $s=n\,\ln(\# \mathcal{D})/\ln(q)$ be the pseudo similarity dimension of $K$. Then $\mathcal{H}_w^s(K)=(\mathcal{E}_{w, s}^+(\mu))^{-1}$, where $\mu$ is defined by (\ref{e-1-1}) and $\mathcal{E}_{w, s}^+(\mu)$ is the upper $s$-density of $\mu$ w.r.t. $w(x)$  defined by
$$
\mathcal{E}_{w,s}^+(\mu) = \lim\limits_{r\rightarrow\infty}\sup\limits_{\text{diam}_w U \ge r>0}\frac{\mu(U)}{(\text{diam}_w U)^s},
$$
where the supremum is over all convex sets $U$ with $diam_w U\ge r>0$ w.r.t. w(x).
\end{theorem}

We will divide the proof of Theorem \ref{th1.4} into two cases, (i) and (ii),
with the case (i) corresponding to the situation where the IFS $\{f_d\}_{d\in\mathcal{D}}$ satisfies the OSC
and the case (ii) where it does not.

\vspace{0.3cm}

It follows from  Theorem \ref{th1.3} that if the IFS $\{f_d\}_{d\in{\mathcal D}}$ satisfies the OSC, then $K: = K(A, \mathcal {D})$ is an $s$-set w.r.t. $w(x)$.  By  analyzing the upper convex $s$-density w.r.t. $w(x)$ of points in $K$, we have the following expression of $\mathcal{H}_w^s(K)$.

\begin{lemma}\label{th1.5}
Let $K:= K(A,\mathcal{D})$ be the self-affine set associated with an
IFS $\{f_d\}_{d\in\mathcal{D}}$ satisfying the OSC.
Let $s=n\,\ln(\# \mathcal{D})/\ln(q)$ and let $\sigma$ be the invariant measure supported on $K$ satisfying
$$
\int f \ d\sigma=\frac{1}{\#\mathcal{D}}\sum\limits_{d\in\mathcal{D}}\int f\circ f_d \ d\sigma
$$
for any compactly supported continuous function $f$ on $\mathbb{R}^n$.
Then,  for any $r_0>0$,
$$
(\mathcal{H}_w^s(K))^{-1}=\sup\limits_{0<\text{diam}_w U\le r_0}\frac{\sigma(U)}{(\text{diam}_w U)^s},
$$
 where the supremum is taken over all convex sets $U$ with $U\bigcap K\ne\emptyset$ and
 $0<\text{diam}_w U\le r_0$.
\end{lemma}

For case (i), Theorem \ref{th1.4} will follow from Lemma \ref{th1.5}
after we prove that
$$
\mathcal{E}_{w, s}^+(\mu)=\sup\limits_{0<\text{diam}_w U\le r_0}\frac{\sigma(U)}{(\text{diam}_w U)^s}.
$$
For case (ii), we show $\mathcal{E}_{w, s}^+(\mu)= \infty$  by using the third equivalent condition in Theorem \ref{th1.3}.

\vspace{0.3cm}

The paper is organized as follows. In Section 2, we collect some definitions and some known results on pseudo norm, pseudo Hausdorff dimension and pseudo Hausdorff measures
that we will use. In Section 3, we prove Theorem \ref{th1.3}. Some properties of  upper convex $s$-density w.r.t. $w(x)$ of points in $K(A,\mathcal{D})$ and the upper $s$-density of $\mu$ w.r.t. $w(x)$ are investigated respectively in Section 4 and in Section 5. In Section 6,  Lemma \ref{th1.5} and Theorem \ref{th1.4} are proved.

\section{Preliminaries}

In this section, we recall the notions of pseudo norm and pseudo Hausdorff measure defined in \cite{HL}.
and collect some known results about these that we will
use later.

\vspace{0.3cm}

Let $A\in M_n(\mathbb{R})$ be expanding with  $q:=|\det A| \in {\mathbb R}$. We can assume without loss of generality that $A$ has the property that $\|x\| \le \|Ax\|$ and equality holds only for $x=0$, where the norm $\|\cdot\|$ is the Euclidean norm, since $\|\cdot\|$ in ${\mathbb R}^n$ can be renormed with an equivalent norm $\|\cdot\|^{'}$ so that $\|x\|^{'}< \|Ax\|^{'}$ for all $0\ne x\in{\mathbb R}^n$ \cite{LWG}. He and Lau \cite{HL} introduced a \emph{pseudo norm $w(x)$} associated with $A$ as follows:
\begin{itemize}
\item For $0 < \delta < 1/2$, choose a positive function $\phi_{\delta}(x)\in C^{\infty}(\mathbb{R}^n)$ with support in $B_{\delta}:=B(0, \delta)$ (the closed ball centered at $0$ with radius $\delta$) such that $\phi_{\delta}(x)=\phi_{\delta}(-x)$ and $\int \phi_{\delta}(x) \ dx = 1$.

\item Let $V = AB_1\setminus B_1$ and $h(x) = \chi_V * \phi_{\delta}(x)$. Define
\begin{eqnarray}\label{e-2-1}
w(x) = \sum\limits_{j = -\infty}^{\infty} q^{-\frac{j}{n}}h(A^j x), \ x\in\mathbb{R}^n.
\end{eqnarray}
\end{itemize}
Note that $V$ is an annular region by our convention that $\|x\| < \|Ax\|$ for $x\ne 0$. It is clear that ${\mathbb R}^n\setminus \{0\} = \bigcup\limits_{k\in\mathbb{Z}}A^k V$, where the union is disjoint.

\vspace{0.3cm}

\begin{proposition}[\cite{HL}]\label{t-2-1}
The $w(x)$ defined in (\ref{e-2-1}) is a $C^{\infty}$ function on $\mathbb{R}^n$ and
satisfies
\begin{enumerate}[(i)]
\item $w(x) = w(-x)$, $w(x) = 0\Leftrightarrow x = 0$;
\item $w(Ax) = q^{1/n}w(x)$, $x\in\mathbb{R}^n$;
\item there exists an integer $p>0$ such that for each $x\in{\mathbb R}^n$, the sum in (\ref{e-2-1}) has at most $p$ non-zero terms and $\alpha \le w(x) \le pq^{p/n}, x\in V$, where $\alpha = \inf_{x\in V} h(x) >0$.
\end{enumerate}
\end{proposition}

 He and Lau \cite{HL} showed that the pseudo norm $w(x)$ is comparable with the Euclidean norm $\|x\|$ through $\lambda_{max}$ and $\lambda_{min}$, the maximal and minimal moduli of the eigenvalues of $A$.  For more details about the properties of $w(x)$ and its relationship with the Euclidean norm, please refer to \cite{HL, CGV, L}.

\vspace{0.3cm}

\begin{proposition}[\cite{HL}]\label{t-2-2}
Let $A\in M_n({\mathbb R})$ be an expanding matrix with $|\det A| = q$ and let $w(x)$ be a pseudo norm associated with $A$. Then for any $0 < \epsilon < \lambda_{min} -1$, there exists $C>0$ (depending on $\epsilon$) such that
\begin{eqnarray*}
C^{-1}\|x\|^{\ln q/(n\ln (\lambda_{max}+\epsilon))} \le w(x) \le C\|x\|^{\ln q/(n\ln (\lambda_{min}-\epsilon))}, \ \ \|x\|>1,
\end{eqnarray*}

\begin{eqnarray*}
C^{-1}\|x\|^{\ln q/(n\ln (\lambda_{min}-\epsilon))} \le w(x) \le C\|x\|^{\ln q/(n\ln (\lambda_{max}+\epsilon))}, \ \ \|x\|\le 1.
\end{eqnarray*}
\end{proposition}

\vspace{0.3cm}

Unlike Euclidean norm, the triangle inequality is not satisfied for pseudo norm any more. However, we have the following inequality  instead.

\begin{lemma}[\cite{HL}]\label{t-2-3}
There exists $\beta > 0$ such that for any $x, y \in {\mathbb R}^n$,
$$w(x + y) \le \beta \max\{w(x), w(y)\}.$$
\end{lemma}

\vspace{0.3cm}

Furthermore, we can modify Lemma \ref{t-2-3} into the following lemma, which will be used in  Section 5.

\begin{lemma}\label{t-2-4}
 For any $\epsilon>0$, there is a positive number $\lambda_\epsilon>1$ such that for
 any $x_1, x_2\in \mathbb{R}^n$ with $w(x_2)>\lambda_\epsilon w(x_1)$,
 $w(x_1+x_2)<(1+\epsilon)w(x_2)$ holds.
\end{lemma}
\begin{proof}
 Let $V=AB_1\setminus B_1$. Denote $\theta=\max\{\|x\|: x\in V\}$ and $V_1=\bigcup_{x\in V}B(x,1)$.
 Obviously, $w\in C(\overline{V_1})$ since $w\in C^\infty(\mathbb{R}^n)$.
 So, for any $\epsilon>0$, there exists a number
 $\delta$ with $0<\delta<1$ such that $w(z_1)-w(z_2)<\alpha\, \epsilon$
 whenever $z_1, z_2\in V_1$ with $\|z_1-z_2\|\leq \delta$, where $\alpha=\inf_{x\in V}h(x)$
 as introduced in Proposition 2.1. Choose $\lambda_\epsilon>1$ large enough such that
$$\frac{n\,\ln(\lambda_\epsilon\, \alpha/(p\,q^{p/n}))}{\ln q}\geq \frac{-\ln (\delta/\theta)}{\ln\lambda_{\min}},$$
where $p, q$ are the same as in Proposition \ref{t-2-1}.
For any $x_1, x_2\in\mathbb{R}^n$ with $w(x_2)>\lambda_\epsilon w(x_1)$, without loss of generality, assume $x_1\neq 0$ and write $x_1=A^{l_1}y_1$ and $x_2=A^{l_2}y_2$ with
$l_1,l_2\in \mathbb{Z}$ and $y_1,y_2\in V$. It is easy to check that $w(x_i)=q^{{l_i}/{n}}\,w(y_i)$ for $i=1,2$, and hence
$$q^{(l_2-l_1)/{n}}>\lambda_{\epsilon}w(y_1)/w(y_2)\geq \lambda_{\epsilon}\alpha/(p\,q^{p/n}),$$
since $\alpha\leq w(y_i)\leq p\,q^{p/n}$ for $i=1,2$ by Proposition 2.1 (iii). This gives that
\begin{equation*}
l_2-l_1>\frac{n\,\ln(\lambda_{\epsilon}\,\alpha/(p\,q^{p/n}))}{\ln q}
\end{equation*}
and thus
$
l_2-l_1> \frac{-\ln (\delta/\theta)}{\ln\lambda_{\min}}>0.
$
Hence
$$\|A^{l_1-l_2}y_1\|=\|(A^{-1})^{l_2-l_1}y_1\|\leq
\lambda_{\min}^{l_1-l_2}\,\theta<\delta.
$$
So we have
$$w(x_1+x_2)=w(A^{l_2}(A^{l_1-l_2}y_1+y_2))=
q^{l_2/n}\,w(A^{l_1-l_2}y_1+y_2)<q^{l_2/n}\,(w(y_2)+\alpha\,\epsilon)$$
 since $y_1,y_2\in V$ and $\|A^{l_1-l_2}y_1\|<\delta$, and thus
 $$w(x_1+x_2)< (1+\epsilon)\,q^{l_2/n}\, w(y_2) =(1+\epsilon)\,w(x_2).$$
\end{proof}

Next, we come to the definition of pseudo Hausdorff measure and pseudo Hausdorff dimension. For a given set $E\subset {\mathbb R}^n$, the \emph{diameter of $E$ w.r.t. $w(x)$} is defined by
$$\text{diam}_w E = \sup\{w(x-y): x, y\in E\}.$$
 A collection of  sets $\{U_i\}_{i=1}^{\infty}$ in ${\mathbb R}^n$ is called a $\delta$-cover of $E\subset {\mathbb R}^n$ w.r.t. $w(x)$ if $E\subseteq \bigcup\limits_{i=1}^{\infty}\,U_i$ and $\text{diam}_{w} U_i\le\delta$.
 Such a collection is called an open $\delta$-cover of $E$ if $U_i$ is open for all $i\ge 1$.
For  $E\subset\mathbb{R}^n$ and $s\ge 0$, $\delta>0$, define
$$
\mathcal{H}_{w,\delta}^s(E)=\inf\Big\{\sum\limits_{i=1}^{\infty}(\text{diam}_{w}U_i)^s:\,\,\{U_i\}_{i=1}^{\infty} \ \ \text{is a} \ \delta\text{-cover of} \ E \ \text{w.r.t.} \ w(x)\Big\}.
$$
Since $\mathcal{H}_{w,\delta}^s(E)$ is increasing when $\delta$ tends to $0$, we can define the  \emph{$s$-dimensional Hausdorff measure of $E$ w.r.t. $w(x)$} (the \emph{$s$-dimensional pseudo Hausdorff measure of $E$}) by
$$
{\mathcal H}_{w}^{s}(E)=\lim\limits_{\delta\to 0}\mathcal{H}^s_{w,\delta}(E)=\sup\limits_{\delta>0}\mathcal{H}_{w,\delta}^s(E).
$$
It is direct to  see that  $\mathcal{H}_{w}^s$ is a Borel measure on  ${\mathbb R}^n$. By Proposition \ref{t-2-1} (ii), it is easy to obtain that

\begin{eqnarray}\label{e-2-2}
 \mathcal{H}_{w}^s(AE) =  q^{s/n}\mathcal{H}_{w}^s(E).
\end{eqnarray}
As usual, we define the  \emph{Hausdorff dimension of $E$ w.r.t. $w(x)$} ( the \emph{pseudo Hausdorff dimension of $E$}) to be  the quantity
$$\dim _H^{w} E = \inf\{s: \mathcal{H}_{w}^{s}(E) = 0\} = \sup\{s: \mathcal{H}_{w}^{s}(E) = \infty\}.$$

\vspace{0.2cm}

 This setup gives a convenient estimation  of the classical Hausdorff dimension and makes a self-affine set have a structure as a self-similar set since the pseudo norm defined in terms of $A$ absorbs the non-uniform contractivity from $A$.

\vspace{0.2cm}

\begin{theorem}[\cite{HL}]\label{t-2-5}
Let $A\in M_n(\mathbb{R})$ be an expanding matrix with $|\det A| = q \in {\mathbb R}$ and let $w(x)$ be a pseudo norm associated with $A$. Then for any subset $E\subset {\mathbb R}^n$,
\begin{eqnarray*}
\frac{\ln q}{n\ln \lambda_{max}} \dim_H^w E \le  \dim_H E \le \frac{\ln q}{n\ln \lambda_{min}}\dim_H^w E,
\end{eqnarray*}
where $\lambda_{max}, \lambda_{min}$ denote the maximal and minimal moduli of the eigenvalues of $A$, and $\dim_H E$ is the classical Hausdorff dimension of $E$.
\end{theorem}
It follows immediately that
 $\dim_H^w E = \dim_H E$ when $\lambda_{max} = \lambda_{min}$. This includes the special case that $A$ is a similar matrix.

\vspace{0.3cm}

\section{Proof of Theorem \ref{th1.3}}

In the following, let $A\in M_n(\mathbb{R})$ be expanding with $|\det A| = q$ and $0\in\mathcal{D}\subset\mathbb{R}^n$ be a digit set. Let $K:=K(A,\mathcal{D})$ be a self-affine set associated with $(A,\mathcal{D})$. We always assume that $w(x)$ is a pseudo norm associated with $A$.

\vspace{0.3cm}

 He and Lau \cite{HL} proved the direction ``OSC $\Rightarrow 0<\mathcal{H}_w^s(K)<\infty$" for the self-affine case.

\begin{theorem}[\cite{HL}]\label{t-3-1}
Suppose that the IFS $\{f_d\}_{d\in\mathcal{D}}$ satisfies the OSC. Then
$\dim_H^w K = s:=n\,\ln(\# \mathcal{D})/\ln(q)$ and $0<\mathcal{H}_w^s(K)<\infty$.
\end{theorem}

In particular, if $A$ is a  similar matrix with scaling factor $\rho>1$, then
$s: = \ln(\# \mathcal{D})/\ln(\rho)$  is the similarity dimension of the self-similar set $K(A, {\mathcal D})$. For consistency, we call $s: = n\,\ln(\# \mathcal{D})/\ln(q)$  the \emph{pseudo similarity dimension} of the self-affine set $K(A, {\mathcal D})$.

\vspace{0.3cm}

To prove the other direction ``$0<\mathcal{H}_w^s(K)<\infty \Rightarrow$ OSC", Lemma \ref{t-3-2}  and Lemma \ref{t-2-6} below are needed. It is well-known (\cite{H}) that the IFS $\{f_d\}_{d\in\mathcal{D}}$ determines a unique Borel probability measure $\sigma$ supported on the set $K(A,\mathcal{D})$ satisfying
\begin{eqnarray}\label{e-3-1}
\int f \ d\sigma=\frac{1}{\#\mathcal{D}}\sum\limits_{d\in\mathcal{D}}\int f\circ f_d \ d\sigma,
\end{eqnarray}
for any compactly supported continuous function $f$ on $\mathbb{R}^n$. We say that $\sigma$ has \emph{no overlap} if $\sigma(f_d(K)\cap f_{d^{'}}(K) )= 0$ for $d \ne d^{'}\in {\mathcal D}$. Lemma \ref{t-3-2} and its proof show that if  the self-affine set $K$ has positive pseudo Hausdorff measure associated with the dimension $s: = n\,\ln(\# \mathcal{D})/\ln(q)$, then the invariant measure $\sigma$  has no overlap.

\vspace{0.3cm}

\begin{lemma}\label{t-3-2}
Suppose that $0<\mathcal{H}_w^s(K)<\infty$ with $s: = n\,\ln(\# \mathcal{D})/\ln(q)$ and $\sigma$ is a self-affine measure defined in (\ref{e-3-1}). Then
$$\sigma = (\mathcal{H}_w^s(K))^{-1}\mathcal{H}_w^s \restriction K,$$ (i.e. $\sigma$ is the restriction of $\mathcal{H}_w^s$ to $K$ normalized so as to give $\sigma (K) = 1$).
\end{lemma}
\begin{proof}
For any Borel subset $E\subset {\mathbb R}^n$ and $d\in {\mathcal D}$, we have
\begin{eqnarray*}
\mathcal{H}_w^s(f_d^{-1}(E)) = \mathcal{H}_w^s(AE - d) = \mathcal{H}_w^s (AE) = q^{s/n}\mathcal{H}_w^s (E) = (\#{\mathcal D})\,\mathcal{H}_w^s(E).
\end{eqnarray*}
Similarly, $\mathcal{H}_w^s(f_d(E)) = \frac{1}{\#{\mathcal D}}\mathcal{H}_w^s(E)$. Then, we have
\begin{eqnarray*}
\mathcal{H}_w^s(K)
& = &\mathcal{H}_w^s(\bigcup\limits_{d\in {\mathcal D}}f_d(K))
\le \sum\limits_{d\in {\mathcal D}}\mathcal{H}_w^s(f_d(K))\\
& = &\#{\mathcal D} \cdot\frac{1}{\#{\mathcal D}} \mathcal{H}_w^s(K) = \mathcal{H}_w^s(K).
\end{eqnarray*}
This implies that $\mathcal{H}_w^s(f_d(K)\cap f_{d^{'}}(K) = 0$ for $d \ne d^{'}\in {\mathcal D}$ since $0 < \mathcal{H}_w^s(K) <\infty$. Then for any Borel set $E$,
\begin{eqnarray*}
\mathcal{H}_w^s(E\cap K) = \sum\limits_{d\in {\mathcal D}}\mathcal{H}_w^s(E\cap f_d(K))
= \sum\limits_{d\in {\mathcal D}} \frac{1}{\#{\mathcal D}} \mathcal{H}_w^s(f_d^{-1}(E)\cap K).
\end{eqnarray*}
This proves that  $\mathcal{H}_w^s \restriction K$  is invariant for the IFS $\{f_d\}_{d\in\mathcal{D}}$
and thus the probablility measure $(\mathcal{H}_w^s(K))^{-1}\mathcal{H}_w^s \restriction K$ coincides
with $\sigma$ as this last measure is unique.
\end{proof}

\vspace{0.3cm}

For $E, F\subset {\mathbb R}^n$ and $z\in{\mathbb R}^n$, we let
\begin{eqnarray*}
D(E, F)= \inf\{d(x,y): x\in E, y\in F\}\quad \text{and}\quad D(z,E)=D(\{z\},E),
\end{eqnarray*}
where $d$ denotes the distance induced by the Euclidean norm.
The \textit{Hausdorff distance} between  compact sets $E, F\subset {\mathbb R}^n$ is denoted by $D_{\mathcal{H}}(E,F)$ and defined by
\begin{eqnarray*}
D_{\mathcal{H}}(E, F)= \max\{\sup\limits_{x\in E}D(x,F), \sup\limits_{y\in F}D(E,y)\}.
\end{eqnarray*}
Denote $\text{Comp}({\mathbb R}^n)$ the set of compact subsets in $\mathbb{R}^n$. Then Blaschke selection Theorem \cite{BP}  implies that

\begin{theorem}[\cite{BP}]\label{t-3-3}
$(\text{Comp}({\mathbb R}^n), D_{\mathcal{H}})$ is a compact metric space.
\end{theorem}

\vspace{0.3cm}

We use the pseudo norm to replace the Euclidean norm and let
\begin{eqnarray*}
D_w(E, F)= \inf\{d_w(x,y):=w(x-y): x\in E, y\in F\}.
\end{eqnarray*}
Define the \textit{Hausdorff distance w.r.t. $w(x)$} between  compact sets $E$ and $F$ in ${\mathbb R}^n$ by
 $$
 D_{{\mathcal H},w} (E, F) = \max\{\sup\limits_{x\in E}D_w(x,F), \sup\limits_{y\in F}D_w(E,y)\}.
 $$
   Denote
  $
  U_w(x,\epsilon): =\{y\in{\mathbb R}^n: d_w(x,y) < \epsilon\}
  $
 to be the open $\epsilon$-neighborhood of $x\in {\mathbb R}^n$ w.r.t. $w(x)$ and
   $U_w(F,\epsilon) = \bigcup \{U_w(x,\epsilon): x\in F\}$.
   Let $f_1, f_2, \dotsc, f_N$ be the IFS associated with the expansive matrix
   $A\in M_n({\mathbb R})$ and the digit set ${\mathcal D} =
   \{d_1, d_2, \dotsc, d_N\}\subset{\mathbb R}^n$. Let $\Sigma = \{1, 2, \dotsc, N\}$
   and $\Sigma ^m = \{(i_1 i_2 \dotsc i_m): 1\le i_j \le N\}$  for $m\geq 1$.
   Write $\Sigma^*=\bigcup_{m\geq 0}\Sigma^m$ with $\Sigma^0:=\emptyset$.
   For $\mathbf{i} = (i_1 i_2 \dotsc i_m)$ and $\mathbf {j} = (j_1 j_2 \dotsc j_k)$
   in $\Sigma^*$, we use the notation $\mathbf {ij}$ for the element
   $(i_1 i_2 \dotsc i_m j_1 j_2 \dotsc j_k)\in\Sigma^* $,
   and say that $\mathbf{i}$ and $\mathbf{j}$ are incomparable if there exists no
   $\mathbf{k}$ such that $\mathbf{i} = \mathbf{jk}$ or $\mathbf{j} = \mathbf{ik}$.
   It follows from Proposition \ref{t-2-1} (ii) that for any $i\in\Sigma$
 \begin{eqnarray}\label{e-3-2}
 w(f_i(x)- f_i(y)) = q^{-\frac{1}{n}}w(x - y).
  \end{eqnarray}
  Let $r = q^{-\frac{1}{n}}$. For
  $\mathbf{i} \in\Sigma^m$, $m\geq 1$,
  the length of $\mathbf{i}$ is denoted by $|\mathbf{i}| = m$.
  Define
$$
f_{\mathbf{i}} = f_{i_1} \circ f_{i_2} \dotsm \circ f_{i_m}, \ \ K_{\mathbf{i}} = f_{\mathbf{i}}(K) \ \  \text{and} \ \ r_{\mathbf{i}} = r^{|\mathbf{i}|} = q^{-\frac{m}{n}}.
$$
It is obvious that, for any $m\ge 1$,
$K=\bigcup_{\textbf{i}\in\Sigma^m}K_{\mathbf{i}}$.

\vspace{0.3cm}

 According to Lemma \ref{t-3-2}, it is direct to get the following result.

\

\begin{corollary}\label{c-3-0}
Suppose that $0<\mathcal{H}_w^s(K)<\infty$ with $s: = n\,\ln(\# \mathcal{D})/\ln(q)$. Then $\mathbf{i}, \mathbf{j}\in\Sigma^*$ are incomparable if and only if ${\mathcal H}_w^s(K_{\mathbf{i}}\cap K_{\mathbf{j}}) = \emptyset$.
\end{corollary}

\vspace{0.3cm}

Also if we admit only open sets in the covers of $E$, then  $\mathcal{H}_{w,\delta}^s(E)$ (also $\mathcal{H}_{w}^s(E)$) does not change.

\vspace{0.3cm}

\begin{lemma}\label{t-2-6}
For  $E\subset\mathbb{R}^n$ and $s\ge 0$, $\delta>0$, define
$$
\widetilde{\mathcal{H}}_{w,\delta}^s(E)=\inf\Big\{\sum\limits_{i=1}^{\infty}(\text{diam}_{w}U_i)^s:\,\, \{U_i\}_{i=1}^{\infty} \ \ \text{is an open} \ \delta-\text{cover of} \ E \ \text{w.r.t.} \ w(x)\Big\}.
$$
Then $\tilde{\mathcal{H}}_{w,\delta}^s(E) = \mathcal{H}_{w,\delta}^s(E)$.
\end{lemma}
\begin{proof}
It is obvious that $\mathcal{H}_{w,\delta}^s(E)\le \widetilde{\mathcal{H}}_{w,\delta}^s(E)$.
For any $\epsilon > 0$, by the definition of $\mathcal{H}_{w,\delta}^s(E)$, there exists a $\delta$-cover $\{U_i\}_{i=1}^{\infty}$ of $E$ w.r.t. $w(x)$ such that
\begin{eqnarray*}
\mathcal{H}_{w,\delta}^s(E) \ge \sum\limits_{i=1}^{\infty}(\text{diam}_{w}U_i)^s - \epsilon.
\end{eqnarray*}
Denote $U(U_i, 1) = \{y\in{\mathbb R}^n: \|y-x\|< 1 \ \text{for some} \ x\in U_i\}$ to be the open $1$-neighborhood of $U_i$.
For the above $\epsilon > 0$, by using $w(x)\in C(\overline{U(U_i, 1)})$, there exists $\delta_i>0$ such that $|w(x)-w(y)|<\text{diam}_w(U_i)\epsilon$ whenever $\|x-y\|\le\delta_i$ and $x,y\in \overline{U(U_i, 1)}$. Take $\delta_i' = \min\{\delta_i, 1\}$ and $V_i = U(U_i, \frac{\delta_i'}{2})$. Then $U_i\subset V_i\subset U(U_i, 1)$ and $V_i$ is open.
For any $z_1, z_2\in V_i$, by the definition of $V_i$, there exist $x_1, x_2\in U_i$ such that $\|x_j - z_j\|\le \frac{\delta_i'}{2}, j=1, 2$. This and $w(x)\in C(\overline{V_i})$ imply that
\begin{eqnarray}\label{e-3-0}
w(z_1 - z_2) \le w(x_1 -x_2) + \text{diam}_w(U_i)\epsilon \le \text{diam}_w(U_i) + \text{diam}_w(U_i)\epsilon < (1+\epsilon)\delta.
\end{eqnarray}
It follows from (\ref{e-3-0}) that $\text{diam}_w(V_i)\le (1+\epsilon)\text{diam}_w(U_i)< (1+\epsilon)\delta$ since $z_1, z_2\in V_i$ are arbitrary. Using the definition of $\widetilde{\mathcal{H}}_{w,\delta}^s$,
\begin{eqnarray*}
\widetilde{\mathcal{H}}_{w,(1+\epsilon)\delta}^s(E)&\le& \sum(\text{diam}_w V_i)^s \le (1+\epsilon)^s(\text{diam}_w U_i)^s \\
&\le& (1+\epsilon)^s(\mathcal{H}_{w,\delta}^s(E) + \epsilon).
\end{eqnarray*}
Letting $\epsilon\rightarrow 0$, one can get $\widetilde{\mathcal{H}}_{w,\delta}^s(E) \le \mathcal{H}_{w,\delta}^s(E)$.
\end{proof}

\vspace{0.3cm}

\begin{theorem}\label{t-3-4}
If $0<\mathcal{H}_w^s(K(A, \mathcal{D}))<\infty$ with $s: = n\,\ln(\# \mathcal{D})/\ln(q)$ , then the IFS $\{f_d\}_{d\in{\mathcal D}}$ satisfies the OSC.
\end{theorem}

\begin{proof}
Let $x > 0$. By the definition of ${\mathcal H}_w^s(K)$ and Lemma \ref{t-2-6}, there
exists a sequence of open sets $\{U_i\}_{i\ge 1}$ such that
$$U:=\bigcup\limits_{i=1}^{\infty} U_i \supset K \ \text{and} \
\sum\limits_{i=1}^{\infty} (\text{diam}_w U_i)^s \le (1+x^s)\mathcal{H}_w^s(K).$$
{\bf Claim 1:} Denote $\delta = D_w(K, U^c),$ where $U^c$ denotes the complement of $U$. Then for all incomparable $\mathbf{i}, \mathbf{j}$ with $r_{\mathbf{j}} > xr_{\mathbf{i}}$, we have $D_{{\mathcal H},w}(K_{\mathbf{i}}, K_{\mathbf{j}})\ge \delta r_{\mathbf{i}}$.
 \begin{proof}Suppose that Claim 1 does not hold.
 Then there exist a pair $\mathbf{i}$, $\mathbf{j}$ with
 $r_{\mathbf{j}}>xr_{\mathbf{i}}$ and
 $D_{\mathcal{H},w}(K_{\mathbf{i}}, K_{\mathbf{j}})<\delta r_{\mathbf{i}}$. Since clearly
$
 D_w(K_{\mathbf{i}}, (f_{\mathbf{i}}(U))^c) = \delta r_{\mathbf{i}},
 $
 we get $$K_{\mathbf{j}}\subset U_w(K_{\mathbf{i}},
 \delta r_{\mathbf{i}}) \subset f_{\mathbf{i}}(U).$$ This implies that
\begin{eqnarray*}
\mathcal{H}_w^s(K)\,r_{\mathbf{i}}^s\,(1+x^s) & < & \mathcal{H}_w^s(K)\,(r_{\mathbf{i}}^s + r_{\mathbf{j}}^s)
= \mathcal{H}_w^s(K_{\mathbf{i}}) + \mathcal{H}_w^s(K_{\mathbf{j}})\\
& = & \mathcal{H}_w^s(K_{\mathbf{i}}\cup K_{\mathbf{j}})
\le \sum\limits_{i=1}^{\infty} (\text{diam}_w f_{\mathbf{i}}(U_i))^s \\
& = & \sum\limits_{i=1}^{\infty} r_{\mathbf{i}}^s\, (\text{diam}_w U_i)^s \le \mathcal{H}_w^s(K)\,
r_{\mathbf{i}}^s\,(1+x^s),
\end{eqnarray*}
which is a contradiction. (The second to the last inequality follows from the fact that $K_{\mathbf{i}}\cup K_{\mathbf{j}}\subset f_{\mathbf{i}}(U)$ and the second equality is obtained from  Corollary \ref{c-3-0}).
\end{proof}

For $0<b<1$, we set
$I_b=\{\textbf{i}\in \Sigma^*: r^{|\textbf{i}|}\leq b<r^{|\textbf{i}|-1}\}$.
The elements of $I_b$ are obviously
incomparable and satisfy $K = \bigcup_{\mathbf{i}\in I_b} K_{\mathbf{i}}$.

\vspace{0.2cm}

 Fix $0 < \varepsilon < \min\{\text{diam}_w K, \beta\,\text{diam}_w K,
(\beta\,\text{diam}_w K)^2, \lambda_{\min}-1\}$, where $\beta$ satisfies the inequality in Lemma \ref{t-2-3}
and $\lambda_{\min}$ is the minimal moduli of the eigenvalues of $A$. For $\textbf{k}\in \Sigma^*$, denote $G_{\mathbf{k}} = U_w (K_{\mathbf{k}},\varepsilon r_{\mathbf{k}})$. Note that for any $k\ge 1$, the pair $(A, A^{-k}{\mathcal D})$ can determine a self-affine set $A^{-k}K$ if $K$ is determined by the pair $(A,{\mathcal D})$ and the IFS $\{f_d\}_{d\in{\mathcal D}}$ satisfies the OSC if and only if $\{f_{A^{-k}d}\}_{d\in{\mathcal D}}$ satisfies the OSC. To simplify the notations, WLOG we can assume that $\text{diam}_w K$ is small enough such that $\text{diam}_w G_{\mathbf{k}} < 1$ for any $\textbf{k}\in \Sigma^*$ since we can always use $A^{-k}K$ and $\{f_{A^{-k}d}\}_{d\in{\mathcal D}}$ instead of $K$ and $\{f_d\}_{d\in{\mathcal D}}$ if $\text{diam}_w K$ is not small enough.

\vspace{0.3cm}

\noindent {\bf Claim 2:} Denote $
I(\mathbf{k}) = \{\mathbf{i}\in I_{\text{diam}_w G_{\mathbf{k}}}: K_{\mathbf{i}}\cap G_{\mathbf{k}}\ne \emptyset\},
$
and $\gamma = \sup\limits_\textbf{k} \# I(\mathbf{k})$. Then $\gamma < \infty$.
\begin{proof}
 For the  given $\varepsilon > 0$, let $C_i$ and $\alpha_i$, $i=1, 2$, be the number as in Proposition \ref{t-2-2}
 satisfying the inequality that $\|x - y\| \le (C_i d_w(x,y))^{\alpha_i}$ for $\|x-y\|>1$ and $\|x-y\|\le 1$ respectively. Take $C = C_1$ and $\alpha = \alpha_1$ if $(C_1 \beta^3(\text{diam}_w K)^2)^{\alpha_1} \ge (C_2 \beta^3(\text{diam}_w K)^2)^{\alpha_2}$ and if not, we take $C = C_2$ and $\alpha = \alpha_2$.
 Let $B$ be the closed  $(C \beta^3(\text{diam}_w K)^2)^{\alpha}$-neighborhood of $K$, i.e.
 $B=\{x\in \mathbb{R}^n: D(x,K)\leq (C \beta^3(\text{diam}_w K)^2)^{\alpha}\}$. Then for any $\mathbf{k}\in \Sigma^*$, it holds that
\begin{equation}\label{e-3-3}
f_{\mathbf{k}}^{-1}(K_{\mathbf{i}}) \subset B,  \ \forall \ \mathbf{i}\in I(\mathbf{k}).
\end{equation}
In fact, noticing that $K_{\mathbf{i}}\cap G_{\mathbf{k}}\ne \emptyset$ if $\mathbf{i}\in I(\mathbf{k})$, for any $y\in K_{\mathbf{i}}$, it follows from the definition of $d_w$ and Lemma \ref{t-2-3} that,
$$
D_w(y, K_{\mathbf{k}}) \le \beta \max\{d_w(y,z), D_w(z, K_{\mathbf{k}})\} \le \beta \max\{\text{diam}_w K_{\mathbf{i}}, \varepsilon r_{\mathbf{k}}\},
$$
where $z$ is any point in $K_{\mathbf{i}}\cap G_{\mathbf{k}}$.
This gives that
\begin{eqnarray}\label{e-3-4}
D_w(f_{\mathbf{k}}^{-1}(y), K) \le \beta \max\{r_{\mathbf{k}}^{-1}r_{\mathbf{i}}\text{diam}_w K, \varepsilon \}.
\end{eqnarray}
On the other hand, if  $\mathbf{i}\in I(\mathbf{k})$, then $\mathbf{i}\in I_{\text{diam}_w G_{\mathbf{k}}}$ and thus we have $r_{\mathbf{i}} \leq \text{diam}_w G_{\mathbf{k}}$ by the definition of $I_{\text{diam}_w G_{\mathbf{k}}}$. Next, we will utilize Lemma \ref{t-2-3} to give an estimation on $\text{diam}_w G_{\mathbf{k}}$. Let $z_1, z_2\in G_{\mathbf{k}}$. Then there exist $x_1, x_2\in  K_{\mathbf{k}}$ satisfying that $d_w(z_i,x_i)<\varepsilon r_{\mathbf{k}}$ for $i=1, 2$. By Lemma \ref{t-2-3}, we obtain
\begin{eqnarray*}
d_w(z_1,z_2) &=& w(z_1-x_1+x_1-x_2+x_2-z_2) \\
&\le& \beta\max\{w(z_1-x_1), w(x_1-x_2+x_2-z_2)\}\\
&\le& \beta\max\{w(z_1-x_1), \beta\max\{w(x_1-x_2), w(x_2-z_2)\}\}\\
&\le& \beta\max\{\varepsilon r_{\mathbf{k}}, \beta\max\{r_{\mathbf{k}}\text{diam}_w K, \varepsilon r_{\mathbf{k}}\}\}\\
&\le& \beta^2 r_{\mathbf{k}}\text{diam}_w K.
\end{eqnarray*}
The last inequality is obtained by the restriction of $\varepsilon$. This and $r_{\mathbf{i}} \leq \text{diam}_w G_{\mathbf{k}}$ give $r_{\mathbf{i}} \leq \beta^2 r_{\mathbf{k}}\text{diam}_w K$.
Substituting this into (\ref{e-3-4}), one can get $D_w(f_{\mathbf{k}}^{-1}(y), K) \le \beta^3(\text{diam}_w K)^2$. Then by using Proposition \ref{t-2-2}, we have

\begin{eqnarray*}
D(f_{\mathbf{k}}^{-1}(y), K) &\le&  \begin{cases}(C_1 D_w(f_{\mathbf{k}}^{-1}(y), K))^{\alpha_1}, \ \text{if} \  D(f_{\mathbf{k}}^{-1}(y), K) >1, \\(C_2 D_w(f_{\mathbf{k}}^{-1}(y), K))^{\alpha_2}, \ \text{if} \  D(f_{\mathbf{k}}^{-1}(y), K) \le 1
\end{cases}\\
&\le& (C \beta^3(\text{diam}_w K)^2)^{\alpha},
\end{eqnarray*} which proves (\ref{e-3-3}).

\

\noindent  Since for any $\mathbf{i}, \mathbf{j} \in \Sigma^m$, $r_{\mathbf{i}} = r_{\mathbf{i}} = r^m$. Then $r_{\mathbf{j}} \geq r_{\mathbf{i}} r$ holds. We may apply  Claim 1 for $x = r$ to get $\delta > 0$ such that $$D_{{\mathcal H},w}(K_{\mathbf{i}}, K_{\mathbf{j}}) \ge \delta r_{\mathbf{i}}\ge \delta  r_\textbf{k}r\text{diam}_wG$$ for any distinct $\mathbf{i}, \mathbf{j} \in I(\mathbf{k})$, where $G=U_w (K,\epsilon)$. Hence,
\begin{eqnarray*}
D_{{\mathcal H},w}(f_{\mathbf{k}}^{-1}(K_{\mathbf{i}}),
f_{\mathbf{k}}^{-1}(K_{\mathbf{j}}))\ge \delta\, r\, \text{diam}_w G
\end{eqnarray*}
and
\begin{eqnarray*}
D_{{\mathcal H}}(f_{\mathbf{k}}^{-1}(K_{\mathbf{i}}),
f_{\mathbf{k}}^{-1}(K_{\mathbf{j}}))\ge (C^{\prime}\,\delta \,r \,\text{diam}_w G)^{\alpha^{\prime}},
\end{eqnarray*}
 with some positive $C^{\prime}, \alpha^{\prime}$ for all  $\mathbf{i}, \mathbf{j} \in I(\mathbf{k})$ by Proposition \ref{e-2-2}.
By Theorem \ref{t-3-3},  $\# I(\mathbf{k})$ is bounded by the maximal number of  compact subsets of $B$ which are $(C^{\prime}\delta r\text{diam}_wG)^{\alpha^{\prime}}$-separated in the Hausdorff metric, which is obviously independent of $\textbf{k}\in\Sigma^*$.
\end{proof}

\noindent {\bf Claim 3:} Choose $\mathbf{k}$ such that $\gamma = \# I(\mathbf{k})$. Then for any $\mathbf{j}\in \Sigma^*$,
$
I(\mathbf{jk}) = \{\mathbf{ji}: \mathbf{i}\in I(\mathbf{k})\}.
$
\begin{proof}
By maximality, we only need to prove ``$\supset$". This is clear since $\emptyset \ne K_{\mathbf{i}}\cap G_{\mathbf{k}}$ implies
\begin{eqnarray*}
\emptyset &\ne& f_{\mathbf{j}}(K_{\mathbf{i}}\cap G_{\mathbf{k}}) = f_{\mathbf{j}}(K_{\mathbf{i}})\cap f_{\mathbf{j}}(G_{\mathbf{k}}) = K_{\mathbf{ji}}\cap f_{\mathbf{j}} (U_w(K_{\mathbf{k}},\varepsilon r_{\mathbf{k}}))\\
&=& K_{\mathbf{ji}}\cap U_w(K_{\mathbf{jk}},\varepsilon r_{\mathbf{jk}}) = K_{\mathbf{ji}}\cap G_{\mathbf{jk}}.
\end{eqnarray*}
\end{proof}
\noindent {\bf Claim 4:} $D_w(K_{i\mathbf{ik}}, K_{j})\ge \varepsilon r_{i\mathbf{ik}}$ for any $j \ne i$ and any $\mathbf {i}\in\Sigma^*$.
\begin{proof} For any word $j\textbf{l}$ with $j \ne i$, Claim 3 implies that $j\textbf{l} \notin I(i\mathbf{ik})$.
 By the definition of $I(i\mathbf{ik})$, for $j\textbf{l}\in I_{\text{diam}_w G_{i\mathbf{ik}}}$, $K_{j\textbf{l}}\cap G_{i\mathbf{ik}}=\emptyset$. Hence,
 $D_w(K_{i\mathbf{ik}}, K_{j\textbf{l}})\ge \varepsilon r_{i\mathbf{ik}}$. Noticing that
$$K_{j}\subseteq \bigcup\{K_{j\textbf{l}}: j\textbf{l}\in I_{\text{diam}_w G_{i\mathbf{ik}}}\},$$ then Claim 4 follows.
\end{proof}
\noindent{\bf Claim 5:} For $\textbf{i}\in\Sigma^*$, denote $G_{\mathbf{i}}^* = U_w (K_{\mathbf{i}},\beta^{-1}\varepsilon r_{\mathbf{i}})$.   Then
$U = \bigcup\limits_{\mathbf{j}\in\Sigma^*} G_{\mathbf{jk}}^*$ gives the OSC.
\begin{proof}
Clearly, $U$ is open and $K_{\mathbf{k}}\subset G_{\mathbf{k}}^* \subset U$. For each $i$,
\begin{eqnarray*}
f_i(U) = \bigcup\limits_{\mathbf{j}\in\Sigma^*} f_i(G_{\mathbf{jk}}^*) = \bigcup\limits_{\mathbf{j}\in\Sigma^*} G_{i\mathbf{jk}}^* \subset U.
\end{eqnarray*}
For $i\ne j$, $f_i(U)\cap f_j(U) = \emptyset$. Indeed, if not, there exist $\mathbf{i}, \mathbf{j}$ such that
$G_{i\mathbf{ik}}^*\cap G_{j\mathbf{jk}}^* \ne \emptyset$. Let $y\in G_{i\mathbf{ik}}^*\cap G_{j\mathbf{jk}}^*$. Then there exist $y_1\in K_{i\mathbf{ik}}$ and $y_2\in K_{j\mathbf{jk}}$ such that
$w(y - y_1) < \beta^{-1}\varepsilon r_{i\mathbf{ik}}$ and $w(y - y_2) < \beta^{-1}\varepsilon r_{j\mathbf{jk}}$. Without loss of generality, we assume that $r_{i\mathbf{ik}}\geq r_{j\mathbf{jk}}$. Then we have $w(y_1 - y_2) < \varepsilon r_{i\mathbf{ik}}$. Hence, $D_w(K_{i\mathbf{ik}}, K_j) < \varepsilon r_{i\mathbf{ik}}$, which contradicts Claim 4.
\end{proof}
This completes the proof of Theorem \ref{t-3-4}.
\end{proof}

There is another equivalent condition for the OSC provided by He and Lau in \cite{HL}.

\begin{theorem}[\cite{HL}]\label{t-3-5}
Let $A\in M_n(\mathbb{R})$ be expanding and let ${\mathcal D}\subset {\mathbb R}^n$ be a digit set. Then the IFS $\{f_d\}_{d\in{\mathcal D}}$ satisfies the OSC if and only if  $\# {\mathcal D}_M = (\#{\mathcal D})^M$ and ${\mathcal D}_{\infty}$ is a uniformly discrete set.
\end{theorem}

Theorem \ref{t-3-4} together with Theorem \ref{t-3-5} and Theorem \ref{t-3-1} imply Theorem \ref{th1.3}.


\section{The upper convex density  w.r.t. $w(x)$}

In this section, we introduce the notion of $s$-sets w.r.t. the pseudo norm $w(x)$, and study  the upper convex density of an $s$-set w.r.t. $w(x)$ at certain points.  These are definitions analogous to those
corresponding to the Euclidean norm. (See, for example, Section 2 in \cite{F}.)

\vspace{0.3cm}

A subset $E\subset\mathbb{R}^n$ is called an \emph{$s$-set} ($0\le s\le n$) w.r.t. $w(x)$ if
$E$ is $\mathcal{H}_w^s$-measurable and $0<\mathcal{H}_w^s(E)<\infty$.
The \emph{upper convex $s$-density} of an $s$-set $E$ w.r.t. $w(x)$ at $x$ is defined as
\begin{eqnarray*}
D_{w,c}^s(E,x) = \lim\limits_{r\to 0}\sup_{0<\text{diam}_w U\le r, x\in U}\frac{\mathcal{H}_w^s(E\cap U)}{(\text{diam}_w U)^s},
\end{eqnarray*}
where the supremum is over all convex sets $U$ with $x\in U$ and $0<\text{diam}_w U\le r$, and the limit exists obviously.
We have the following result similar to Theorem 2.2 and Theorem 2.3 in \cite{F}.

\begin{theorem}\label{t-4-1}
If $E$ is an $s$-set w.r.t. $w(x)$ in $\mathbb{R}^n$, then $D_{w,c}^s(E,x)=1$ at $\mathcal{H}_w^s$-almost all $x\in E$ and $D_{w,c}^s(E,x)=0$ at $\mathcal{H}_w^s$-almost all $x\in E^c$.
\end{theorem}
We will prove Theorem \ref{t-4-1} by showing that $D_{w,c}^s(E,x)=0$ at $\mathcal{H}_w^s$-almost all $x\in E^c$ (Lemma \ref{t-4-4}) and $D_{w,c}^s(E,x)=1$ at $\mathcal{H}_w^s$-almost all $x\in E$ (Lemma \ref{t-4-5}) respectively. We need an analogue of Vitali covering theorem \cite{F} and the following lemma. We should mention that the sets encountered in the following can always be represented in terms of known $\mathcal{H}_w^s$-measurable sets using combinations of $\overline{\lim}$, $\underline{\lim}$, countable unions and intersections. So without explicit mention in this section, we always assume that the sets involved are $\mathcal{H}_w^s$-measurable.


\begin{lemma}\label{t-4-2}
Let $E\subset {\mathbb R}^n$ be ${\mathcal H}_w^s$-measurable with ${\mathcal H}_w^s(E) < +\infty$ and let $\varepsilon > 0$. Then there exists $\rho > 0$, depending only on $E$ and $\varepsilon$, such that for any collection of Borel sets $\{U_i\}_{i=1}^{\infty}$ with $0 < \text{diam}_w U_i \le \rho$, we have
\begin{eqnarray*}
{\mathcal H}_w^s(E\cap \bigcup\limits_i U_i) < \sum_i (\text{diam}_w U_i)^s + \varepsilon.
\end{eqnarray*}
\end{lemma}
\begin{proof}
By the definition that ${\mathcal H}_w^s = \lim\limits_{\delta\rightarrow 0}{\mathcal H}_{w,\delta}^s$, we may choose $\rho > 0$ such that
\begin{eqnarray}\label{e-4-1}
\mathcal{H}_w^s(E) \le \sum\limits_{i}\text{diam}_w(W_i) + \varepsilon/2
\end{eqnarray}
for any $\rho$-cover $\{W_i\}$ of $E$ w.r.t. $w(x)$.
Given Borel sets $\{U_i\}$ with $0 < \text{diam}_w(U_i)\le \rho$, by the definition of ${\mathcal H}_w^s$, we can find a $\rho$-cover $\{V_i\}$ of $E\setminus \bigcup\limits_{i}U_i$ w.r.t. $w(x)$ satisfying
\begin{eqnarray*}
\mathcal{H}_w^s(E\setminus \bigcup\limits_{i}U_i) + \varepsilon/2 > \sum\limits_{i}\text{diam}_w(V_i).
\end{eqnarray*}
Then $\{U_i\}\cup\{V_i\}$ is a $\rho$-cover of $E$ w.r.t. $w(x)$, and using (\ref{t-4-1}), we have
\begin{eqnarray*}
\mathcal{H}_w^s(E) \le \sum\limits_{i}\text{diam}_w(U_i) + \sum\limits_{i}\text{diam}_w(V_i) + \varepsilon/2.
\end{eqnarray*}
Hence,
\begin{eqnarray*}
\mathcal{H}_w^s(E\cap \bigcup\limits_{i}U_i) &=& \mathcal{H}_w^s(E) - \mathcal{H}_w^s(E\setminus \bigcup\limits_{i}U_i)\\
&<& \sum\limits_{i}\text{diam}_w(U_i) + \sum\limits_{i}\text{diam}_w(V_i) + \varepsilon/2 - \sum\limits_{i}\text{diam}_w(V_i) + \varepsilon/2\\
&=& \sum\limits_{i}\text{diam}_w(U_i) + \varepsilon.
\end{eqnarray*}

\end{proof}

\

A collection of sets ${\mathcal V}$ is called \textit{a Vitali class} for $E$ w.r.t. $w(x)$ if for each $x\in E$ and
 $\delta > 0$, there exists $U\in {\mathcal V}$ with $x\in U$ and $0 < \text{diam}_w U \le \delta$.

\begin{theorem}[Vitali covering theorem]\label{t-4-3}

\

\begin{enumerate}[(a)]
\item Let $E$ be an ${\mathcal H}_w^s$-measurable subset of ${\mathbb R}^n$ and let ${\mathcal V}$ be a
Vitali class of closed sets for $E$ w.r.t. $w(x)$. Then we may select a (finite or countable) disjoint sequence $U_i$
from ${\mathcal V}$ such that either $\sum_i (\text{diam}_w U_i)^s = \infty$ or ${\mathcal H}_w^s (E\setminus \bigcup_i U_i) = 0$.

\item If ${\mathcal H}_w^s(E) < +\infty$, then for any given $\varepsilon > 0$, we may also require that
$${\mathcal H}_w^s(E) \le \sum_i (\text{diam}_w U_i)^s + \varepsilon.$$
\end{enumerate}
\end{theorem}
\begin{proof}{\bf(a).} Fix $\rho > 0$. We may assume that $\text{diam}_w U \le \rho$ for all $U \in {\mathcal V}$.  Let $U_1\in{\mathcal V}$ and $U_1\cap E\ne \emptyset$. We choose $U_i$, $i\geq 2$ inductively.
Suppose that $U_1, \dotsc, U_m$ have been chosen, and let
\begin{eqnarray*}
d_m = \sup\{\text{diam}_w U: U\in{\mathcal V} \ \text{and} \ U\cap U_i = \emptyset, i = 1, 2, \dotsc m\}.
\end{eqnarray*}
Note that $\{d_m\}_{m\ge 1}$ is decreasing. If $d_m = 0$, then $E\subset \bigcup\limits_{i=1}^m U_i$.
Indeed, if there existed a point $x\in E\setminus \bigcup\limits_{i=1}^m U_i$, then, letting
\begin{eqnarray*}
\delta_x =\frac{1}{2}\inf \{{w(x-y), y\in\bigcup\limits_{i=1}^m U_i}\} > 0,
 \end{eqnarray*}
we could find  $U\in\mathcal{V}$ such that $x\in U$ and $0 < \text{diam}_w U < \delta_x$,
contradicting the fact that $d_m = 0$. So (a) follows and the process terminates. Otherwise, let $U_{m+1}\in {\mathcal V}$ be a set satisfying $U_{m+1}\cap (\bigcup\limits_{i=1}^m U_i) = \emptyset$ and $\text{diam}_w (U_{m+1})\ge \frac{1}{2}d_m$.

  Suppose that the process continues indefinitely and that
  $\sum (\text{diam}_w U_i)^s < \infty$. For each $i$,
  let $B_i$ be a  pseudo  ball  centered in $U_i$  with radius
  $2\,\beta\,\text{diam}_w(U_i)$, where $\beta$ is the constant
  in Lemma \ref{t-2-3}. We claim that for every $k \geq 1$,
\begin{eqnarray}\label{e-4-2}
E\setminus \bigcup\limits_{i=1}^k U_i \subset \bigcup\limits_{i=k+1}^{\infty} B_i.
\end{eqnarray}
In fact, for $x\in E\setminus \bigcup\limits_{i=1}^k U_i$, there exists $U\in{\mathcal V}$ with $x\in U$ and $U\cap(\bigcup\limits_{i=1}^k U_i) = \emptyset$. By the assumption that $\sum (\text{diam}_w U_i)^s < \infty$, we obtain that  $\lim\limits_{i\rightarrow \infty} \text{diam}_w U_i = 0$.
Hence, we have $\text{diam}_w U > 2\,\text{diam}_w U_{\ell}\ge d_{\ell-1}$ for some $\ell\ge k+2$.
If $U\cap U_j=\emptyset$ for $k<j<\ell$ and thus for $1\le j\le\ell-1$,
it would follow that
$$
\text{diam}_w U > 2\,\text{diam}_w U_{\ell}\ge d_{l-1}\ge\text{diam}_w U,
$$
a contradiction. Let thus $i$ be the smallest integer $j$ with $k<j<\ell$ such that
 $U\cap U_j\ne \emptyset$. Since $U\cap U_j=\emptyset$ for $1\le j\le i-1$,
 we have
 $$
 \text{diam}_w U \le d_{i-1}\le 2\,\text{diam}_w U_{i}.
 $$
 By elementary geometry, we have  $U\subset B_i$ and (\ref{e-4-2}) follows.

Thus, if $\delta > 0$,
\begin{eqnarray*}
{\mathcal H}_{w,\delta}^s(E\setminus \bigcup\limits_{i=1}^{\infty}U_i) \le {\mathcal H}_{w,\delta}^s(E\setminus \bigcup\limits_{i=1}^k U_i) \le \sum\limits_{i=k+1}^{\infty}(\text{diam}_w B_i)^s \le  2^s\beta^{2s}\sum\limits_{i=k+1}^{\infty}(\text{diam}_w U_i)^s,
\end{eqnarray*}
provided that $k$ is large enough to ensure that $\text{diam}_w B_i \le \delta$ for $i > k$. Hence, for all $\delta > 0$,
$${\mathcal H}_{w,\delta}^s(E\setminus \bigcup\limits_{i=1}^{\infty}U_i) = 0.$$
So ${\mathcal H}_w^s(E\setminus \bigcup\limits_{i=1}^{\infty}U_i) = 0$ which proves (a).

\noindent {\bf (b).} Suppose that $\rho$ chosen at the beginning of the proof is the number corresponding to $\varepsilon$ and $E$ given in Lemma \ref{t-4-2}. If $\sum_i (\text{diam}_w U_i)^s = +\infty$, then (b) is obvious. Otherwise, by (a) and Lemma \ref{t-4-2}, we obtain
\begin{eqnarray*}
{\mathcal H}_w^s(E) &=& {\mathcal H}_w^s(E\setminus \bigcup\limits_{i=1}^{\infty}U_i) + {\mathcal H}_w^s(E\cap (\bigcup\limits_{i=1}^{\infty}U_i))\\
 &=& 0 + {\mathcal H}_w^s(E\cap \bigcup\limits_{i=1}^{\infty}U_i) <  \sum\limits_{i=1}^{\infty} (\text{diam}_w U_i)^s + \varepsilon.
\end{eqnarray*}
\end{proof}

\begin{lemma}\label{t-4-4}
If $E$ is an $s$-set w.r.t.~$w(x)$ in $\mathbb{R}^n$, then $D_{w,c}^s(E,x)=0$ for $\mathcal{H}_w^s$-almost all $x\in E^c$.
\end{lemma}

\begin{proof}
Fix $\alpha > 0$, we show that the measurable set $F = \{x\notin E: D_w^s(E, x) > \alpha\}$ has zero pseudo Hausdorff measure.
By the regularity of ${\mathcal H}_w^s$, for any given $\delta > 0$, there exists a closed set $E_1\subset E$
such that ${\mathcal H}_w^s(E\setminus E_1) < \delta$. For $\rho > 0$, let
\begin{eqnarray*}\label{e-4-3}
{\mathcal V} = \{ U \ \text{closed \& convex}: 0<\text{diam}_w U\le\rho,\,\,
U\cap E_1 = \emptyset,\,\, {\mathcal H}_w^s(E\cap U) > \alpha (\text{diam}_w U)^s\}.
\end{eqnarray*}
Then ${\mathcal V}$ is a  Vitali class of closed sets for $F$ w.r.t. $w(x)$. It follows from Theorem \ref{t-4-3} (a) that we can find a disjoint sequence of sets $\{U_i\}$ in ${\mathcal V}$ with either $\sum (\text{diam}_w U_i)^s = +\infty$ or ${\mathcal H}_w^s(F\setminus \bigcup\limits_i U_i) = 0$. However, by the definition of ${\mathcal V}$,
\begin{eqnarray*}
\sum\limits_i (\text{diam}_w U_i)^s & < & \frac{1}{\alpha} \sum\limits_i {\mathcal H}_w^s(E \cap U_i)
= \frac{1}{\alpha}{\mathcal H}_w^s(E\cap \bigcup\limits_i U_i) \\
& \le & \frac{1}{\alpha}{\mathcal H}_w^s(E\setminus E_1) < \frac{\delta}{\alpha} < +\infty.
\end{eqnarray*}
This implies that ${\mathcal H}_w^s(F\setminus \bigcup\limits_i U_i) = 0$, and thus we have
\begin{eqnarray*}
{\mathcal H}_{w,\rho}^s(F) &\le& {\mathcal H}_{w,\rho}^s(F\setminus \bigcup\limits_i U_i)
+ {\mathcal H}_{w,\rho}^s(F\cap \bigcup\limits_i U_i)\\
& \le & {\mathcal H}_w^s(F\setminus \bigcup\limits_i U_i) + \sum\limits_i (\text{diam}_w U_i)^s < \frac{\delta}{\alpha} + 0.
\end{eqnarray*}
This is true for any $\delta > 0$ and any $\rho > 0$. So ${\mathcal H}_w^s(F) = 0$.
\end{proof}

\begin{lemma}\label{t-4-5}
If $E$ is an $s$-set  w.r.t.~$w(x)$ in $\mathbb{R}^n$, then $D_{w,c}^s(E,x)=1$ at $\mathcal{H}_w^s$-almost all $x\in E$.
\end{lemma}
\begin{proof}
Firstly, we use the definition of pseudo Hausdorff measure w.r.t. $w(x)$ to show that $D_{w,c}^s(E,x)\ge 1$ a.e. in $E$. Take
$\alpha < 1$ and $\rho > 0$. Let
\begin{eqnarray*}\label{e-4-1-a}
F = \{x\in E: \mathcal{H}_w^s(E\cap U)\le \alpha
(\text{diam}_w U)^s \text { for all convex} \ U \ \text{with} \ x\in U \ \text{and} \ \text{diam}_w U\le \rho\}.
\end{eqnarray*}
 For any $\varepsilon > 0$, we may find a $\rho$-cover of $F$ by convex sets $\{U_i\}$
such that
\begin{eqnarray*}
\sum (\text{diam}_w U_i)^s < \mathcal {H}_w^s (F) + \varepsilon.
\end{eqnarray*}
Hence, assuming that each $U_i$ contains some points of $F$ and using the definition of $F$, we obtain
\begin{eqnarray*}
\mathcal {H}_w^s (F) \le \sum_i \mathcal {H}_w^s (F\cap U_i) \le \sum_i \mathcal {H}_w^s (E\cap U_i)
\le \alpha \sum_i(\text{diam}_w U_i)^s \le \alpha\,(\mathcal {H}_w^s (F)+\varepsilon).
\end{eqnarray*}
Since $\alpha < 1$ and the outer inequality holds for all $\varepsilon > 0$, we conclude that $\mathcal {H}_w^s (F) = 0$.
We may define such $F$ for any $\rho > 0$. So $D_{w,c}^s(E,x) \ge \alpha$ for a.e. $x\in E$ by the definition. This is true for all $\alpha < 1$, so we conclude that $D_{w,c}^s(E,x) \ge 1$ a.e. in $E$.

\vspace{0.2cm}

\noindent Secondly, we use a  Vitali method to show that $D_{w,c}^s(E,x)\le 1$ a.e. in $E$. Given $\alpha > 1$,
let $F:=\{x\in E: D_{w,c}^s(E,x) > \alpha\}$ be a measurable subset of $E$ and let
$$F_0=\{x\in F: D_{w,c}^s(E\setminus F,x) = 0\}.$$
Then ${\mathcal H}_w^s(F\setminus F_0) = 0$ by Lemma \ref{t-4-4}. By the definition of upper convex  $s$-density, for  $x\in F_0$, we have
\begin{eqnarray*}\label{e-4-3-a}
D_{w,c}^s(F,x)\ge D_{w,c}^s(E,x) - D_{w,c}^s(E\setminus F,x) > \alpha.
\end{eqnarray*}
Thus,
\begin{eqnarray}\label{e-4-3-b}
{\mathcal V} = \{ U \ \text{closed \& convex}: {\mathcal H}_w^s(F\cap U) > \alpha\,(\text{diam}_w U)^s\}
 \end{eqnarray}
 is a  Vitali class for $F_0$. Then, by Theorem \ref{t-4-3} (b), for any given $\varepsilon > 0$, we can find a disjoint sequence $\{U_i\}_i$ in ${\mathcal V}$ such that ${\mathcal H}_w^s (F_0)\le \sum\limits_{i}(\text{diam}_w U_i)^s + \varepsilon$. By (\ref{e-4-3-b}), we obtain that
 \begin{eqnarray*}
 {\mathcal H}_w^s(F) = {\mathcal H}_w^s(F_0) \le \sum\limits_{i}(\text{diam}_w U_i)^s + \varepsilon < \frac{1}{\alpha} \sum\limits_{i}{\mathcal H}_w^s(F\cap U_i) + \varepsilon \le \frac{1}{\alpha}{\mathcal H}_w^s(F) + \varepsilon.
 \end{eqnarray*}
 This inequality holds for any $\varepsilon > 0$. Hence, we have ${\mathcal H}_w^s(F) = 0$ if $\alpha > 1$ as required.
\end{proof}

Theorem \ref{t-3-1} implies that if the IFS $\{f_d\}_{d\in\mathcal{D}}$ satisfies the OSC,
then the corresponding self-affine set $K:=K(A,\mathcal{D})$ is an $s$-set w.r.t. $w(x)$,
where $s=\dim_H^w K= n\,\ln(\# \mathcal{D})/\ln(q)$ is the pseudo similarity dimension of $K$. Thus Theorem \ref{t-4-1} can be applied to $K$ directly.

\vspace{0.3cm}

\section{The upper $s$-density of $\mu$ w.r.t. $w(x)$ }

In this section, let $\mu$ be a Borel measure on $\mathbb{R}^n$, we use the pseudo norm $w(x)$ instead of the Euclidean norm to define the \textit{upper $s$-density of $\mu$ w.r.t. $w(x)$}. It will be used to find
a different expression for the pseudo Hausdorff measure of $K(A, {\mathcal D})$.  This is motivated by the connection between the upper $s$-density of $\mu$  in (\ref{e-1-1}) which was first introduced in \cite{FG} and the Hausdorff measure of a self-similar set $K(A, {\mathcal D})$.

\vspace{0.2cm}

\begin{definition}\label{d-5-1}
Let $\mu$ be a  Borel  measure in $\mathbb{R}^n$. The \emph{ upper $s$-density of $\mu$ w.r.t. $w(x)$} is defined by
$$
\mathcal{E}_{w,s}^+(\mu) = \lim\limits_{r\rightarrow\infty}\sup\limits_{\text{diam}_w U \ge r>0}\frac{\mu(U)}{(\text{diam}_w U)^s},
$$
where the supremum is over all compact convex sets $U\subseteq {\mathbb R}^n$ with $\text{diam}_w U\ge r>0$.

\end{definition}

 Let $\mu$ be a Borel measure and let $\sigma$ be a Borel probability measure. The convolution $\mu*\sigma$ is defined  to be the measure so that
\begin{eqnarray*}
\int_{\mathbb{R}^n}\phi(z)\,d(\mu*\sigma)(z):=\int_{\mathbb{R}^n}\int_{\mathbb{R}^n}\phi(x+y)\,d\mu(x)\,d\sigma(y),
\end{eqnarray*}
holds for any compactly supported continuous function $\phi$ on $\mathbb{R}^n$.

\begin{lemma}\label{t-5-1}
Let $\mu$ and $\sigma$ be  two Borel measures on $\mathbb{R}^n$ with  $\sigma$ being a  probability measure. Then
$\mathcal{E}_{w,s}^+(\mu*\sigma)=\mathcal{E}_{w,s}^+(\mu).$
\end{lemma}
\begin{proof}
By the definition of $\mathcal{E}_{w,s}^+(\mu)$ and the convolution of $\mu\ast\sigma$, we get
\begin{eqnarray}\label{e-5-1}
\mathcal{E}_{w,s}^+(\mu*\sigma)&=& \lim\limits_{r\rightarrow\infty}\sup_{\text{diam}_w U\ge r>0}\frac{\mu*\sigma(U)}{(\text{diam}_w U)^s}\nonumber\\
&=&\lim\limits_{r\rightarrow\infty}\sup_{\text{diam}_w U\ge r>0}
\frac{\int_{\mathbb{R}^n}\,\int_{\mathbb{R}^n}\,\chi_{U}(x+y)\,d\mu(x)\,d\sigma(y)}{(\text{diam}_w U)^s}\nonumber\\
&=& \lim\limits_{r\rightarrow\infty}\sup_{\text{diam}_w U\ge r>0}\frac{\int_{\mathbb{R}^n}\,\mu(U-y)\,d\sigma(y)}{(\text{diam}_w U)^s},
\end{eqnarray}
where the supremum is over all convex sets $U\subset{\mathbb R}^n$ with $\text{diam}_w U\ge r>0$. Since $\sigma$ is a Borel probability measure, we have
\begin{eqnarray*}
\lim\limits_{r\rightarrow\infty}\sup_{\text{diam}_w U\ge r>0}
\frac{\int_{\mathbb{R}^n}\,\mu(U-y)\,d\sigma(y)}{(\text{diam}_w U)^s}&\le& \lim\limits_{r\rightarrow\infty}\sup_{\text{diam}_w U\ge r>0}
\sup\limits_{y\in\mathbb{R}^n}\,\frac{\mu(U-y)}{(\text{diam}_w U)^s}\\
&=&\lim\limits_{r\rightarrow\infty}\sup_{\text{diam}_w U\ge r>0}\frac{\mu(U)}{(\text{diam}_w U)^s},
\end{eqnarray*}
which implies that $\mathcal{E}_{w,s}^+(\mu*\sigma)\le \mathcal{E}_{w,s}^+(\mu)$.

For the converse inequality,  fix a real number $R>0$. Let $\epsilon>0$ and
$r\geq \lambda_\epsilon\,\beta^2\,R$ where $\lambda_\epsilon$ is the same as in Lemma \ref{t-2-4} and $\beta$ is defined in Lemma \ref{t-2-3}. For any set $U\subset \mathbb{R}^n$ with $\text{diam}_wU\geq r$, choose a set $\tilde{U}=\bigcup_{y\in B_w(0,R)} (U+y)$. Obviously $U\subset \tilde{U}-y$ for any $y\in B_w(0,R)$, the closed ball centered at $0$ with radius $R$ w.r.t.~$w(x)$.
Moreover, we claim that $\text{diam}_w \tilde{U}\leq (1+\epsilon)\,\text{diam}_wU$. In fact, for any two points $x_1,x_2\in \tilde{U}$, we write $x_i=z_i+y_i$ with $z_i\in U$ and $y_i\in B_w(0,R)$ for $i=1,2$. Then $w(y_1-y_2)\leq \beta R$. If $w(z_1-z_2)>\lambda_\epsilon\, \beta \,R$, then we have
$w(z_1-z_2)>\lambda_\epsilon\, w(y_1-y_2)$, and this gives
$$
w(x_1-x_2)=w((z_1-z_2)+(y_1-y_2))<(1+\epsilon)\,w(z_1-z_2)
$$
by  Lemma \ref{t-2-4}. Otherwise if $w(z_1-z_2)\leq \lambda_\epsilon \beta R$, then we have
$$
w(x_1-x_2)\leq \beta\max\{w(z_1-z_2), w(y_1-y_2)\}\leq \beta\max\{\lambda_\epsilon\beta R, \beta R\}=\lambda_\epsilon\beta^2 R\leq r.
$$
Thus we  have $w(x_1-x_2)\leq (1+\epsilon)\,\text{diam}_w U$ in both cases, which yields the claim since $x_1,x_2$ are arbitrary points in $\tilde{U}$.
Then we have
\[\begin{aligned}
\frac{\int_{B_w(0,R)}\,\mu(U)\,d\sigma(y)}{(\text{diam}_w U)^s}&\leq\frac{\int_{B_{w}(0,R)}\,\mu(\tilde{U}-y)\,d\sigma(y)}{(\text{diam}_w\tilde{U})^s}\cdot\frac{(\text{diam}_w\tilde{U})^s}{(\text{diam}_wU)^s}\\&\leq \frac{\int_{B_{w}(0,R)}\,\mu(\tilde{U}-y)\,d\sigma(y)}{(\text{diam}_w\tilde{U})^s}\cdot(1+\epsilon)^s.
\end{aligned}\]
Hence, we have
\begin{eqnarray*}
&&\lim_{r\rightarrow\infty}\sup_{\text{diam}_wU\geq r>0}\frac{\int_{B_w(0,R)}\,\mu(U)\,
d\sigma(y)}{(\text{diam}_w U)^s}\\
&\leq& \lim_{r\rightarrow\infty}\sup_{\text{diam}_wU'\geq r>0}\frac{\int_{B_{w}(0,R)}\,\mu(U'-y)\,d\sigma(y)}{(\text{diam}_wU')^s}\cdot(1+\epsilon)^s\\
&\leq& \lim_{r\rightarrow\infty}\sup_{\text{diam}_wU'\geq r>0}\frac{\int_{\mathbb{R}^d}\,\mu(U'-y)\,d\sigma(y)}{(\text{diam}_wU')^s}\cdot(1+\epsilon)^s\\
&=&\mathcal{E}_{w,s}^+(\mu*\sigma)\cdot(1+\epsilon)^s.
\end{eqnarray*}
By letting $\epsilon\rightarrow 0$ and $R\rightarrow\infty$, we obtain that $\mathcal{E}_{w,s}^+(\mu)\leq\mathcal{E}_{w,s}^+(\mu*\sigma)$.
\end{proof}

\

\begin{lemma}\label{t-5-2}
Let $\sigma$ be the Borel probability measure supported on $K(A,\mathcal{D})$ which satisfies (\ref{e-3-1}). For  $M\geq 1$,
 define $\mu_M=\sum\limits_{x\in\mathcal{D}_M}\delta_{x}$, then
 for any Borel measurable set $W\subset\mathbb{R}^n$, we have
 $\sigma(A^{-M}W)=\frac{1}{(\#\mathcal{D})^M}\,\left(\mu_M*\sigma\right)(W)$.
\end{lemma}
\begin{proof}
For any Borel measurable set $W\subset\mathbb{R}^n$, we deduce from the identity (\ref{e-3-1}) that
\begin{eqnarray*}
\sigma(A^{-M}W)
&=&\int_{\mathbb{R}^n}\chi_{A^{-M}W}(x) \ d\sigma(x)\\
&=&\frac{1}{(\#\mathcal{D})^M}\sum\limits_{d_1,d_2,\dotsc,d_M\in\mathcal{D}}\int_{\mathbb{R}^n}
\chi_{A^{-M}W}(A^{-M}x+A^{-1}d_1+\dotsm+A^{-M}d_M) \ d\sigma(x)\\
&=&\frac{1}{(\#\mathcal{D})^M}\sum\limits_{d_1,d_2,\dotsc,d_M\in\mathcal{D}}\int_{\mathbb{R}^n}\chi_{W}(x+A^{M-1}d_1+\dotsb+d_M) \ d\sigma(x)\\
&=&\frac{1}{(\#\mathcal{D})^{M}}\int_{\mathbb{R}^n}\chi_W(x) \ d(\sigma*\mu_M)(x)\\
&=&\frac{1}{(\#\mathcal{D})^M}\sigma*\mu_M(W).
\end{eqnarray*}
\end{proof}

\vspace{0.2cm}

\section{ Pseudo Hausdorff measure of self-affine sets}

This section is devoted to proving Theorem \ref{th1.2} by considering the IFS $\{f_d\}_{d\in{\mathcal D}}$ satisfies and does not satisfy  the OSC separately. The following technical lemma is needed. We borrow the technique of its proof from \cite{Olsen08} for the self-similar case.

\begin{lemma}\label{t-6-0}
Let the IFS  $\{f_d\}_{d\in{\mathcal D}}$ satisfy the OSC. Then $\mathcal{H}_w^s(K\cap U) \le (\text{diam}_w U)^s$ for any subset $U$ in ${\mathbb R}^n$.
\end{lemma}
\begin{proof}
We will prove the statement by a contradiction. Assume that there exists a subset $U\subset {\mathbb R}^n$ such that $\mathcal{H}_w^s(K\cap U) > (\text{diam}_w U)^s$. Then we can find some $0 < \kappa <1$ such that
$(1-\kappa)\mathcal{H}_w^s(K\cap U) > (\text{diam}_w U)^s$.
Fix $\delta > 0$ and choose a positive integer $m$ such that $\text{diam}_w f_{\mathbf{i}}(U) \le \delta$ for all words $\mathbf{i}\in\Sigma^ m$, where $\Sigma^ m$ is defined in Section 3. Note that
\begin{eqnarray}\label{e-6-1}
\bigcup_{\mathbf{i}\in\Sigma^m} f_{\mathbf{i}}(K\cap U) \subset K \cap \bigcup_{\mathbf{i}\in\Sigma^m} f_{\mathbf{i}}(U),
\end{eqnarray}
since $f_{\mathbf{i}}(K) \subset \bigcup_{\mathbf{j}\in\Sigma^m} f_{\mathbf{j}}(K) = K$ for each $\mathbf{i}\in\Sigma^m$. By the assumption that the IFS $\{f_d\}_{d\in{\mathcal D}}$ satisfies the OSC, then by using Theorem \ref{t-3-1} and Lemma \ref{t-3-2}, we have
${\mathcal H}_w^s(f_{\mathbf{i}}(K\cap U)) \cap f_{\mathbf{j}}(K\cap U)) = 0$ for distinct $\mathbf{i}, \mathbf{j}\in\Sigma^ m$. Therefore, by (\ref{e-6-1}), we obtain
\begin{eqnarray}\label{e-6-2}
{\mathcal H}_w^s(K\cap \bigcup_{\mathbf{i}\in\Sigma^m} f_{\mathbf{i}}(U)) &\ge& {\mathcal H}_w^s( \bigcup_{\mathbf{i}\in\Sigma^m} f_{\mathbf{i}}(K\cap U)) \nonumber \\
& = & \sum\limits_{\mathbf{i} \in\Sigma^ m}  {\mathcal H}_w^s  (f_{\mathbf{i}}(K\cap U)) = {\mathcal H}_w^s (K\cap U).
\end{eqnarray}
Defining $\eta = \frac{1}{2}\,\kappa\, {\mathcal H}_w^s( K\cap
\bigcup_{\mathbf{i}\in\Sigma^m} f_{\mathbf{i}}( U))$, it  follows from (\ref{e-6-2}) that
$$\eta \ge \frac{1}{2}\,\kappa\, {\mathcal H}_w^s( K\cap U) > \frac{1}{2}\kappa \frac{(\text{diam}_w U)^s}{1-\kappa} > 0.$$
For $\eta > 0$, we can choose a sequence of sets $\{U_i\}_i$ with $\bigcup\limits_i U_i \supseteq K\setminus \bigcup\limits_{\mathbf{i}\in\Sigma ^m} f_{\mathbf{i}}(U)$ and $\text{diam}_w(U_i) < \delta$ such that
\begin{eqnarray}\label{e-6-3}
\sum\limits_{i} (\text{diam}_w U_i)^s &\le&  {\mathcal H}_{w,\delta}^s (K\setminus \bigcup\limits_{\mathbf{i}\in\Sigma^m} f_{\mathbf{i}}(U)) + \eta \nonumber\\
&\le & {\mathcal H}_{w}^s(K\setminus \bigcup\limits_{\mathbf{i}\in\Sigma^m} f_{\mathbf{i}}(U)) + \eta.
\end{eqnarray}
The family $\{f_{\mathbf{i}} (U)\}_{\mathbf{i}\in\Sigma^ m} \cup \{U_i\}_i$ is clearly a $\delta$-cover of $K$ w.r.t. $w(x)$. Using the fact that $\sum\limits_{\mathbf{i}\in\Sigma ^ m} r_{\mathbf{i}}^s =1$ and (\ref{e-6-3}), we obtain that
\begin{eqnarray*}
{\mathcal H}_{w,\delta}(K) & \le & \sum\limits_{\mathbf{i}\in\Sigma^ m} (\text{diam}_w f_{\mathbf{i}}(U))^s + \sum\limits_i (\text{diam}_w U_i)^s \\
&\le & (\text{diam}_w U)^s + {\mathcal H}_{w}^s(K\setminus \bigcup\limits_{\mathbf{i}\in\Sigma^ m} f_{\mathbf{i}}(U)) + \eta  \\
&\le & (1 - \kappa)\, {\mathcal H}_{w}^s(K\cap U) + {\mathcal H}_{w}^s(K\setminus \bigcup\limits_{\mathbf{i}\in\Sigma^ m} f_{\mathbf{i}}(U)) + \eta \\
\end{eqnarray*}
Taking the inequality (\ref{e-6-2}) into account, this yields
\begin{eqnarray*}
{\mathcal H}_{w,\delta}(K)&\le & (1 - \kappa) {\mathcal H}_{w}^s(K\cap \bigcup\limits_{\mathbf{i}\in\Sigma^ m} f_{\mathbf{i}}(U)) + {\mathcal H}_{w}^s(K\setminus \bigcup\limits_{\mathbf{i}\in\Sigma^ m} f_{\mathbf{i}}(U)) + \eta \\
&\le & {\mathcal H}_{w}^s(K) - \kappa\, {\mathcal H}_w^s (K\cap \bigcup\limits_{\mathbf{i}\in\Sigma^ m} f_{\mathbf{i}}(U)) + \eta \\
& = &  {\mathcal H}_{w}^s(K) - \eta \\
&\le & {\mathcal H}_{w}^s(K) - \frac{1}{2}\,\kappa\, {\mathcal H}_{w}^s(K\cap U).
\end{eqnarray*}
Letting $\delta\rightarrow 0$, we get
\begin{eqnarray*}
{\mathcal H}_{w}^s(K)  \le {\mathcal H}_{w}^s(K) - \frac{1}{2}\kappa {\mathcal H}_{w}^s(K\cap U),
\end{eqnarray*}
which is a contradiction since $0 < {\mathcal H}_{w}^s(K) < \infty$ and $\frac{1}{2}\kappa {\mathcal H}_{w}^s(K\cap U)>0$.
\end{proof}

Lemma \ref{t-3-2} shows that if the IFS $\{f_d\}_{d\in\mathcal{D}}$ satisfies the OSC,
then the probability measure $\sigma$ in (\ref{e-3-1})
is a multiple of the restriction of the $s$-dimensional pseudo Hausdorff measure $\mathcal{H}_w^s$ to the set $K$, with $s=\dim_H^w K = n\,\ln(\# \mathcal{D})/\ln(q) $, i.e.
\begin{eqnarray}\label{e-6-4}
\sigma=(\mathcal{H}_w^s(K))^{-1}\mathcal{H}_w^s\restriction K.
\end{eqnarray}

Combining the formula (\ref{e-6-4}), Lemma \ref{t-6-0}, Theorem \ref{t-3-1} and Theorem \ref{t-4-1}, we obtain the following lemma.

\begin{lemma}\label{t-6-1}
Let $K:= K(A,\mathcal{D})$ be a self-affine set and let the IFS $\{f_d\}_{d\in\mathcal{D}}$ satisfy the OSC.
Then for any $r_0>0$,
$$
(\mathcal{H}_w^s(K))^{-1}=\sup\limits_{0<\text{diam}_w U\le r_0}\frac{\sigma(U)}{(\text{diam}_w U)^s},
$$
 where $s$ is the pseudo similarity dimension of $K$, $\sigma$ is defined by (\ref{e-3-1}) and the supremum is over all convex sets $U$ with $U\bigcap K\ne\emptyset$ and $0<\text{diam}_w U\le r_0$.
\end{lemma}
\begin{proof}

By applying Theorem \ref{t-3-1}, $K$ is an $s$-set w.r.t. $w(x)$. From Theorem \ref{t-4-1}, we can pick a point $x\in K$ such that $D_{w,c}^s(K,x)=1$. Then there exists a positive sequence $\{r_n\}_n$ with $r_n\leq r_0$, $r_n\rightarrow 0$ as $n\rightarrow\infty$ such that
$$\sup_{0<\text{diam}_w U\leq r_n, x\in U}\frac{\mathcal{H}_w^s(K\cap U)}{(\text{diam}_w U)^s}-\frac{1}{n}\leq 1\leq\sup_{0<\text{diam}_w U\leq r_n, x\in U}\frac{\mathcal{H}_w^s(K\cap U)}{(\text{diam}_w U)^s}+\frac{1}{n}. $$ For each $n$, there exists a convex set $U_n$ containing $x$ with $0<\text{diam}_w U_n\leq r_n$ such that
$$\sup_{0<\text{diam}_w U\leq r_n, x\in U}\frac{\mathcal{H}_w^s(K\cap U)}{(\text{diam}_w U)^s}\leq \frac{\mathcal{H}_w^s(K\cap U_n)}{(\text{diam}_w U_n)^s}+\frac{1}{n}.$$
Thus $$\frac{\mathcal{H}_w^s(K\cap U_n)}{(\text{diam}_w U_n)^s}-\frac{1}{n}\leq 1\leq \frac{\mathcal{H}_w^s(K\cap U_n)}{(\text{diam}_w U_n)^s}+\frac{2}{n},$$
which yields that $\frac{\mathcal{H}_w^s(K\cap U_n)}{(\text{diam}_w U_n)^s}\rightarrow 1$ as $n\rightarrow\infty$. Moreover, by Lemma \ref{t-6-0}, for each convex set $U$ with $K\cap U\neq \emptyset$, we have $\frac{\mathcal{H}_w^s(K\cap U)}{(\text{diam}_w U)^s}\leq 1$. Hence $\sup_{0<\text{diam}_wU\leq r_0}\frac{\mathcal{H}_w^s(K\cap U)}{(\text{diam}_w U)^s}=1$. By applying the formula  (\ref{e-6-4}) to the above equality,  the lemma follows.
\end{proof}



We have the following representation for the pseudo Hausdorff measure of self-affine sets.

\vspace{0.3cm}

\begin{theorem}\label{t-6-2}
Let $K:=(A,\mathcal{D})$ be a self-affine set and let $s: = n\,\ln(\# \mathcal{D})/\ln(q)$  be the pseudo similarity dimension of $K$. Then $\mathcal{H}_w^s(K)=(\mathcal{E}_{w,s}^+(\mu))^{-1}$, where $\mu$ is defined by (\ref{e-1-1}).
\end{theorem}

\begin{proof} Let us assume first that $\mathcal{H}_w^s(K)>0$ and thus that the OSC holds by Theorem \ref{th1.3}. By Lemma \ref{t-6-1}, it is sufficient  to prove that
$$\mathcal{E}_{w,s}^+(\mu)=\sup\limits_{0<\text{diam}_w U\le r_0}\frac{\sigma(U)}{(\text{diam}_w U)^s}$$
for some $r_0>0$,  where the supremum is over all convex sets $U$ with $U\cap K\ne\emptyset$ and  $0<\text{diam}_w U\le r_0$.

Fix $r_0>0$. It follows from Lemma \ref{t-6-1} that
 $\sup\limits_{0<\text{diam}_w U\le r_0}\frac{\sigma(U)}{(\text{diam}_w U)^s}$ is finite.
Then, for any given $\varepsilon>0$, there exists a convex set $U_0$ with $\text{diam}_wU_0\le r_0$ and  $U_0\cap K\ne\emptyset$ such that
\begin{eqnarray}\label{e-6-7}
\frac{\sigma(U_0)}{(\text{diam}_w U_0)^s}\ge \sup\limits_{0<\text{diam}_w U\le r_0}\frac{\sigma(U)}{(\text{diam}_w U)^s}-\varepsilon.
\end{eqnarray}
For any $N\geq 1$, define $\mu_N=\sum\limits_{d_0,\dotsc,d_{N-1}\in\mathcal{D}}\delta_{d_0+Ad_1+\dotsb+A^{N-1}d_{N-1}}$.
Using Lemma \ref{t-5-1} and Lemma \ref{t-5-2}, we have
\begin{eqnarray}\label{e-6-8}
\frac{\sigma(U_0)}{(\text{diam}_w U_0)^s}&=&\frac{\sigma*\mu_N(A^NU_0)}{(\text{diam}_w (A^NU_0))^s}=
\lim_{N\to \infty}\,\frac{\sigma*\mu_N(A^NU_0)}{(\text{diam}_w (A^NU_0))^s}\nonumber\\
&\le& \lim_{r\to \infty}\sup_{\text{diam}_w U\ge r>0}\frac{\sigma*\mu(U)}{(\text{diam}_w U)^s}\nonumber\\
&=&\mathcal{E}_{w,s}^+(\sigma*\mu)=\mathcal{E}_{w,s}^+(\mu).
\end{eqnarray}
It follows from (\ref{e-6-7}) and (\ref{e-6-8}) that
$$\sup\limits_{0<\text{diam}_w U\le r_0}\frac{\sigma(U)}{(\text{diam}_w U)^s}\le \mathcal{E}_{w,s}^+(\mu)+\varepsilon.$$
By letting $\varepsilon\rightarrow 0$, we get
$$\sup\limits_{0<\text{diam}_w U\le r_0}\frac{\sigma(U)}{(\text{diam}_w U)^s}\le \mathcal{E}_{w,s}^+(\mu).$$

Conversely, for any given convex set $U$, using Lemma \ref{t-5-2}, we have,
\begin{eqnarray*}
\frac{\sigma*\mu(U)}{(\text{diam}_w U)^s}&=&\lim\limits_{N\to\infty}\frac{\sigma*\mu_N(U)}{(\text{diam}_w U)^s}=\lim\limits_{N\to\infty}\frac{\sigma(A^{-N}U)}{(\text{diam}_w (A^{-N} U))^s}\\
&\le&\sup\limits_{0<\text{diam}_w V\le r_0}\frac{\sigma(V)}{(\text{diam}_w V)^s}.
\end{eqnarray*}
Using Lemma \ref{t-5-1} again, we have thus that
$$\mathcal{E}_{w,s}^+(\mu)=\mathcal{E}_{w,s}^+(\mu*\sigma)\le \sup\limits_{0<\text{diam}_w V\le r_0}\frac{\sigma(V)}{(\text{diam}_w V)^s}.$$
Thus we have proved the desired result in the case that $\mathcal{H}_w^s(K)>0$.

\vspace{0.2cm}

 On the other hand, if $\mathcal{H}_w^s(K)=0$, then the
 IFS $\{f_d\}_{d\in\mathcal{D}}$ does not satisfy the OSC by Theorem \ref{th1.3}.
 Thus, by Theorem \ref{th1.3}, either the $(\#{\mathcal D})^M$
 expansions in ${\mathcal D}_M$ are not
 distinct for some  $M$ or ${\mathcal D}_{\infty}$ is not uniformly discrete.
For $z\in\mathbb{R}^n$, we will use
$$I_k (z) = \{y=(y_1, \dotsc, y_n)\in{\mathbb R}^n: |y_i - z_i|\le \frac{k}{2}, i = 1, 2,\dotsc, n\}$$
to denote the cube  centered at $z=(z_1, \dotsc, z_n)\in{\mathbb R}^n$  with side length $k$.

Assume first that there exists some $M$ such that the $(\#{\mathcal D})^M$ expansions in ${\mathcal D}_M$ are not distinct. Then there exists $a\in{\mathcal D}_M$ which can be represented in two different ways in terms  of the digits in ${\mathcal D}$, i.e.
\begin{eqnarray*}
a = \sum\limits_{j=0}^{M-1}\, A^j\,d_j = \sum\limits_{j=0}^{M-1}\, A^j\,d_j^{\prime}, \quad d_j, d_j^{\prime} \in {\mathcal D},
\end{eqnarray*}
with $d_j\ne d_j^{\prime}$  for  at  least  one  $0\le j\le M-1$. Then $a + A^M a$ has at least four distinct expansions in ${\mathcal D}_{2M}$. More generally, for  $k\geq 1$, $\sum\limits_{j=0}^{k-1} A^{Mj}a$ has at least $2^k$ distinct expansions in  ${\mathcal D}_{kM}$. Then, if  $a_k = \sum\limits_{j=0}^{k-1} A^{Mj}a$, then $\mu(\{a_k\}) \ge 2^k$.  Then, for any  $r > 0$,
we have
$$
\frac{\mu(I_r(a_k))}{(\text{diam}_w I_r(a_k))^s} \geq \frac{2^k}{(\text{diam}_w I_r(0))^s}\to \infty,\quad k\to
\infty,
$$
This implies that $\sup\limits_{\text{diam}_w U \ge r>0}\frac{\mu(U)}{(\text{diam}_w U)^s} = \infty$  for any $r>0$, and in particular, ${\mathcal E}_{w,s}^+(\mu) = \infty$.

Next, assume that  $\#\mathcal{D}_M=(\#\mathcal{D})^M$ holds for each $M\geq 1$, but ${\mathcal D}_{\infty}$ is not a uniformly discrete set. Then there exists $M_1\ge 1$ and $x_1, y_1\in {\mathcal D}_{M_1} \subseteq {\mathcal D}_{\infty}$ with $x_1\ne y_1$ such that $\|x_1 - y_1\| < \frac{1}{2}$.  Write $F_1 = \{x_1, y_1\}$ and $w_1 = x_1$. Then $F_1\subset {\mathcal D}_{k_1}\subset {\mathcal D}_{\infty}$ and $\|z_1 - w_1\| < \frac{1}{2}$ for any $z_1\in F_1$. Let $S_1 = 0$.  Inductively, for $k\ge 2$, assume that $M_j, S_j$ and $x_j, y_j\in {\mathcal D}_{M_j}$, $F_j \subset {\mathcal D}_{S_j+M_j}$ have been defined for $1\le j \le k-1$. Let $S_k = \sum\limits_{j=1}^{k-1} M_j$. Choose $M_k$ and  $x_k, y_k\in {\mathcal D}_{M_k}\subset {\mathcal D}_{\infty}$ with $x_k\ne y_k$ and $\|x_k - y_k\| < \frac{1}{2^k\|A\|^{S_k}}$. Write
\begin{eqnarray*}
&& F_k = \{z_1 + A^{S_2}z_2 + \cdots + A^{S_k}z_k: z_i \in \{x_i, y_i\}, 1\le i \le k\}, \\
&& w_k = x_1 + A^{S_2}x_2 + \cdots + A^{S_k}x_k.
\end{eqnarray*}
Then $F_k\subset {\mathcal D}_{S_k+M_k}\subset {\mathcal D}_{\infty}$, $w_k\in {\mathcal D}_{S_k+M_k}$. Thus for any $k\geq 1$, $z\in F_k$,  we have
\begin{eqnarray*}
\|z - w_k\| &=& \|(z_1 - x_1) + A^{S_2}(z_2 - x_2) + \cdots + A^{S_k}(z_k - x_k)\|\\
&\le& \frac{1}{2} + \|A\|^{S_2}\frac{1}{2^2\|A\|^{S_2}} + \cdots \|A\|^{S_k}\frac{1}{2^k\|A\|^{S_k}}\\
&<& 1.
\end{eqnarray*}
This shows that $\mu(I_2(w_k)) \ge 2^k$. Hence, for any $r\ge 2$,  we have $I_2(w_k) \subset I_r(w_k)$ and
$$
\frac{\mu(I_r(w_k))}{(\text{diam}_w I_r(w_k))^s}\geq \frac{2^k}{(\text{diam}_w I_r(0))^s}\to \infty,\quad k\to
\infty.
$$
So ${\mathcal E}_{w,s}^+(\mu) = \infty$ as before.

 Therefore, we always have $\mathcal{H}_w^s(K)=(\mathcal{E}_{w,s}^+(\mu))^{-1}$.
\end{proof}

\end{document}